\documentclass{aptpub}
\authornames{A. BRAVERMAN, J.G. DAI} 
\shorttitle{High order diffusion approximation of the Erlang-C system} 
\usepackage[letterpaper]{geometry}
\usepackage{amsmath,amsfonts}
\usepackage{color,graphicx}
\usepackage[colorlinks,linkcolor=blue,citecolor=cyan,anchorcolor=blue]{hyperref}
\usepackage[numbers,sort]{natbib}
\usepackage{enumerate}
\usepackage{wrapfig}
\newcommand{\R}{\mathbb{R}}

\newcommand{\Z}{\mathbb{Z}}
\newcommand{\EE}{\mathbb{E}}

\newcommand{\lipone}{\text{\rm Lip(1)}}
\newcommand{\Prob}{\mathbb{P}}

\providecommand{\abs}[1]{\left\lvert#1\right\rvert}
\providecommand{\norm}[1]{\lVert#1\rVert}
\usepackage[normalem]{ulem}

\numberwithin{equation}{section}
\begin{document}
\title{High order steady-state diffusion approximation of the Erlang-C system}

\authorone[Cornell University]{Anton Braverman}
\authortwo[Cornell University]{J.G. Dai} 
\emailone{ab2329@cornell.edu}
\emailtwo{jd694@cornell.edu}

\begin{abstract}
In this paper we introduce a new diffusion approximation for the steady-state customer count of the Erlang-C system. Unlike previous diffusion approximations, which use the steady-state distribution of a diffusion process with a constant diffusion coefficient, our approximation uses the steady-state distribution of a diffusion process with a \textit{state-dependent} diffusion coefficient. We show, both analytically and numerically, that our new approximation is an order of magnitude better than its counterpart. To obtain the analytical results, we use Stein's to show that a variant of the Wasserstein distance between the normalized customer count distribution and our approximation vanishes at a rate of $1/R$, where $R$ is the offered load to the system. In contrast, the previous approximation only achieved a rate of $1/\sqrt{R}$. We hope our results motivate others to consider diffusion approximations with state-dependent diffusion coefficients.
\end{abstract}
\keywords{Stein's method; steady-state; diffusion approximation; Erlang-C; high order; convergence rates.}
\ams{60K25}{60F99; 60J60}

\section{Introduction}
Starting with \cite{GurvHuanMand2014}, a number of recent papers \cite{Gurv2014, BravDai2015, BravDaiFeng2015} studied convergence rates for steady-state diffusion approximations. In this paper we focus on the $M/M/n$ system, known as the Erlang-C system, which was also studied in \cite{BravDaiFeng2015}. This system has $n$ homogeneous servers that serve customers in a first-in-first-serve manner. Customers arrive according to a Poisson process with rate $\lambda$, and customer service times are assumed to be i.i.d.\ having exponential distribution with mean $1/\mu$. An important quantity in the system is the offered load, defined as $R = \lambda/\mu$. The customer count process
\begin{align*}
X = \{X(t), t \geq 0\}
\end{align*}
is a continuous time Markov chain (CTMC), where $X(t)$ is the number of customers in the system at time $t$. This CTMC admits a stationary distribution if and only if 
\begin{align*}
R < n,
\end{align*}
which we assume from now on. Let $X(\infty)$ be a random variable having the stationary distribution of $X$, and let $\tilde X(\infty) = (X(\infty) - R)/\sqrt{R}$ be a normalized version. In \cite{BravDaiFeng2015}, the authors presented an approximation for $\tilde X(\infty)$, which we denote here as $Y_0(\infty)$. The random variable $Y_0(\infty)$ corresponds to the steady-state distribution of a diffusion process with \emph{constant} diffusion coefficient $2\mu$, and piece-wise linear drift
\begin{align}
b(x) = \mu \big[(x+\zeta)^- +\zeta\big], \quad \text{ where } \quad  \zeta = (R-n)/\sqrt{R}. \label{eq:bdef}
\end{align}
This random variable has density
\begin{align}
\eta(x) = \kappa \exp\Big({\int_0^x \frac{b(y)}{\mu}dy}\Big), \label{eq:eta}
\end{align}
where $\kappa$ is a normalizing constant. It was shown in \cite{BravDaiFeng2015} that for all $n \geq 1, \lambda > 0$, and $\mu > 0$ satisfying $1 \leq R < n$, the Wasserstein distance
\begin{align}
d_W(\tilde X(\infty), Y_0(\infty)) \equiv \sup_{h \in \lipone} \big| \EE h(\tilde X(\infty)) - \EE h(Y_0(\infty)) \big| \leq \frac{205}{\sqrt{R}}, \label{eq:oldmain}
\end{align}
where 
\begin{align*}
\lipone = \big\{h: \R \to \R \ \big|\  \abs{h(x)-h(y)} \leq \abs{x-y} \text{ for all $x,y\in \R$}\big\}.
\end{align*}
That is, the diffusion approximation error decreases at a rate of $1/\sqrt{R}$. The approximation error bounds presented for the systems studied in \cite{GurvHuanMand2014, Gurv2014, BravDai2015} are of similar magnitudes, decreasing at rates of $1/\sqrt{\lambda}$, where $\lambda$ is the arrival rate to the system of interest in each paper.

In this paper we present a different approximation for $\tilde X(\infty)$. Our approximation is the real-valued random variable $Y(\infty)$, whose density is given by 
\begin{equation}
  \label{eq:stdden}
  \nu(x)= \frac{\kappa}{a(x)} \exp\Big({\int_0^x \frac{2b(y)}{a(y)}dy}\Big),
\end{equation}
where $\kappa>0$ is the normalization constant, $b(x)$ is as in \eqref{eq:bdef}, and 
\begin{align}
& a(x)  = \mu \Big[ 1 + 1(x > -\sqrt{R})\Big(1 -  \frac{ (x + \zeta)^- + \zeta}{\sqrt{R}}\Big)\Big]\label{eq:adef}
\end{align}
This random variable has the steady-state distribution of a diffusion process with drift $b(x)$, and \emph{state dependent} diffusion coefficient $a(x)$.
\begin{theorem}
\label{thm:w2}
There exists a constant $C_{W_2} > 0$ (independent of $\lambda, n$, and $\mu$), such that for all $n \geq 1, \lambda > 0$, and $\mu > 0$ satisfying $1 \leq R < n $,
\begin{equation}
  \label{eq:newmain}
  d_{W_2}(\tilde X(\infty), Y(\infty)) \equiv \sup_{h \in W_2} \big| \EE h(\tilde X(\infty)) - \EE h(Y(\infty)) \big| \le \frac{C_{W_2}}{R},
\end{equation}
where 
\begin{align}
W_2 = \big\{h: \R \to \R\ \big|\  h(x), h'(x) \in \lipone \big\}. \label{eq:spacew2}
\end{align}
\end{theorem}
Theorem~\ref{thm:w2} states that the approximation error of $Y(\infty)$ decreases at a rate of $1/R$, and for this reason we refer to $Y(\infty)$ as the high order approximation. 
This rate is an order of magnitude better than the rates in any of the previously mentioned papers. The class of functions $W_2$ in \eqref{eq:newmain} is not significantly smaller than $\lipone$ in \eqref{eq:oldmain}, meaning that the two statements are comparable. We will see in Appendix~\ref{app:probmet} that $W_2$ is a rich enough class of functions to imply convergence in distribution.

The explicit value of the constant $C_{W_2}$ can be recovered from our proofs, but it would be quite large (on the order of $10^6$), meaning that $Y_0(\infty)$ could potentially be a better approximation than $Y(\infty)$ for small values of $R$. To check that this is not the case, we perform a numerical study to compare the accuracy of both approximations. Some results are presented in Figure~\ref{fig1} and Tables~\ref{tabbenefit} and \ref{tabrates} below, and indicate that $Y(\infty)$ is indeed a better approximation than $Y_0(\infty)$, and that the issue raised above is simply an artifact of our method. 

We mentioned that both $Y_0(\infty)$ and $Y(\infty)$ correspond to diffusion processes. Both of these processes have the same piece-wise linear drift $b(x)$, but the process corresponding to $Y(\infty)$ has a state-dependent diffusion coefficient, whereas the process corresponding to $Y_0(\infty)$ has a constant diffusion coefficient. Using a state-dependent diffusion coefficient is the reason for our improved accuracy. One consequence of our main results is that 
\begin{align}
\big| \EE X(\infty) - \big( R + \EE Y(\infty) \big) \big| \leq \frac{C_{W_2}}{\sqrt{R}}. \label{eq:const_error}
\end{align}
That is, the approximation error of the unscaled mean shrinks at a rate of $1/\sqrt{R}$. The previous approximation in \cite{BravDaiFeng2015} was only able to guarantee a constant gap for the approximation of the mean. 
 
The numerical results in Figure~\ref{fig1} and Tables~\ref{tabbenefit} and \ref{tabrates} show that $Y(\infty)$ consistently outperforms $Y_0(\infty)$. In Table~\ref{tabbenefit} we see that for large or heavily loaded systems, i.e. when $R$ is either large or close to $n$, the approximation $Y_0(\infty)$ performs reasonably well, and the accuracy gained from using $Y(\infty)$ is not as impressive. However, the accuracy gain of $Y(\infty)$ is much more significant for smaller systems with lighter loads. This is further illustrated in Figure~\ref{fig1}, where we approximate the probability mass function (PMF) of $\tilde X(\infty)$ for a small system (even though Theorem~\ref{thm:w2} does not guarantee anything about the approximation of the PMF). In Table~\ref{tabrates} we see that the errors of $Y_0(\infty)$ and $Y(\infty)$ indeed decrease at a rate of $1/\sqrt{R}$ and $1/R$, respectively. Furthermore, the table suggests that the approximation error of the second moment also decreases at a rate of $1/R$, even though \eqref{eq:newmain} does not guarantee this. Numerically, we observed a rate of $1/R$ for higher moments as well. This is not surprising, as there is nothing preventing us from repeating the analysis in this paper for higher moments. Additional numerical results are presented in Appendix~\ref{app:numeric}.

 \begin{figure}[tb]
  \vspace{-1.75in}
 \centerline{\includegraphics[width=120mm,keepaspectratio]{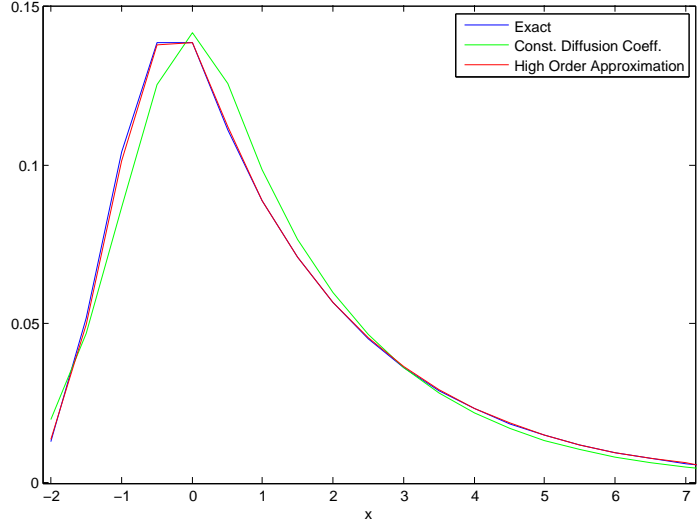}}
  \vspace{-1.75in}
 \caption{The plot above corresponds to a system with $n = 5$ and $R = 4$. The blue line plots $\Prob(\tilde X(\infty) = x)$, where $x = (k-R)/\sqrt{R}$ with $k \in \Z_+$, which is the probability mass function of $\tilde X(\infty)$. The green and red lines plot the approximations $\Prob\Big(x - \frac{1}{2\sqrt{R}} \leq Y_0(\infty) \leq x + \frac{1}{2\sqrt{R}}\Big)$ and $\Prob\Big(x - \frac{1}{2\sqrt{R}} \leq Y(\infty) \leq x + \frac{1}{2\sqrt{R}}\Big)$, respectively. \label{fig1}}
 
 \end{figure}

\begin{table}[tb]
  \begin{center}
   \begin{tabular}{rc|cc|cc }
\multicolumn{6}{c}{$n=5$} \\
$R$ & $\EE \tilde X(\infty)$ & $\big| \EE Y_0(\infty) - \EE \tilde X(\infty)\big|$  &Relative Error & $\big| \EE Y(\infty) - \EE \tilde X(\infty)\big|$& Relative Error\\
\hline
 3        &   0.20  &$5.87\times 10^{-2}$  & 28.69\% & $9.34\times 10^{-3}$ & 4.57\% \\
 4        &  1.11  & $9.91\times 10^{-2}$  & 8.95\% & $1.12\times 10^{-2}$ & 1.08\% \\
 4.9        &  21.04  &$1.28\times 10^{-1}$ & 0.61\% & $1.29\times 10^{-2}$ & 0.06\% \\
 4.95        &  43.39  & $1.29\times 10^{-1}$ & 0.30\% & $1.29\times 10^{-2}$ & 0.03\% \\
 4.99        & 222.26  & $1.30\times 10^{-1}$& 0.06\% & $1.29\times 10^{-2}$ & 0.006\% \\
  \end{tabular}
  \\~\\~\\
  \begin{tabular}{rc|cc|cc }
\multicolumn{6}{c}{$n=100$} \\
$R$ & $\EE \tilde X(\infty)$ & $\big| \EE Y_0(\infty) - \EE \tilde X(\infty)\big|$  &Relative Error & $\big| \EE Y(\infty) - \EE \tilde X(\infty)\big|$& Relative Error\\
\hline
 60 & $2.97\times 10^{-7}$ & $2.73\times 10^{-7}$ & 91.83\% & $5.11\times 10^{-8}$ &17.24\% \\
 80 & $8.79\times 10^{-3}$  & $2.25\times 10^{-3}$ & 25.60\%& $1.03\times 10^{-4}$ & 1.17\% \\
 98 & $3.84$ & $2.85\times 10^{-2}$ & 0.74\% & $7.00\times 10^{-4}$ & 0.02\% \\
99 & $8.78$ & $3.04\times 10^{-2}$ & 0.35\% & $7.26\times 10^{-4}$ & 0.008\% \\
99.8 & $48.74$ & $3.19\times 10^{-2}$ & 0.07\% & $7.46\times 10^{-4}$ & 0.002\% \\
  \end{tabular}
  \end{center}
  \caption{The new approximation $Y(\infty)$ consistently outperforms $Y_0(\infty)$. \label{tabbenefit}}
\end{table}

\begin{table}[h!]
\begin{center}\setlength\tabcolsep{3pt}
\begin{tabular}{rc | c | c | c }
 $n$ & $R$ &$\EE \tilde X(\infty)$ & $\big| \EE \tilde X(\infty) - \EE Y_0(\infty)\big|$ & $\big| \EE \tilde X(\infty) - \EE Y(\infty)\big|$ \\
\hline
5    & 4       & 1.11 & $9.9 \times 10^{-2}$ & $1.2 \times 10^{-2}$ \\
50   & 46.59   & 1.04 & $3.2 \times 10^{-2}$ & $1.2 \times 10^{-3}$ \\
500  & 488.94  & 1.02 & $1.0 \times 10^{-2}$ & $1.2 \times 10^{-4}$  \\
5000 &  4965   & 1.01 & $3.3 \times 10^{-3}$ & $1.2 \times 10^{-5}$  \\
\end{tabular}
\end{center}

\begin{center}\setlength\tabcolsep{3pt}
\begin{tabular}{rc | c | c | c }
 $n$ & $R$ &$\EE (\tilde X(\infty))^2$ & $\big| \EE (\tilde X(\infty))^2 - \EE (Y_0(\infty))^2\big|$ & $\big| \EE (\tilde X(\infty))^2 - \EE (Y(\infty))^2\big|$ \\
\hline
5    & 4       & 6.54 & 1.00  & $6 \times 10^{-2}$ \\
50   & 46.59   & 5.84 &  0.30 & $5.7 \times 10^{-3}$ \\
500  & 488.94  & 5.63 & 0.092 &  $5.6 \times 10^{-4}$  \\
5000 &  4965   & 5.57 & 0.029 &  $5.5 \times 10^{-5}$  \\
\end{tabular}
\end{center}

\caption{As the offered load increases by a factor of $10$, the approximation error of $Y_0(\infty)$ shrinks at a rate of $\sqrt{10}$, whereas the approximation error of $Y(\infty)$ shrinks at a rate of $10$. Similar experiments for moments higher than the second yield consistent results. \label{tabrates}}
\end{table}

To derive our high-order approximation, we use Stein's method \cite{Stei1972,Stei1986}, which is a well studied method for establishing convergence rates that has been widely used
in probability, statistics, and their wide range of applications such
as bioinformatics; see, for example, the survey papers \cite{Ross2011, Chat2014}, the recent book \cite{ChenGoldShao2011} and the
references within. Barbour \cite{Barb1990} has used Stein's
method to establish convergence rates for approximations by a diffusion process.  In this paper, we rely on the framework that was developed in \cite{BravDai2015} for steady-state diffusion approximations of CTMCs. A simpler illustration of this framework is provided in \cite{BravDaiFeng2015}. 

Our new approximation arises naturally from the use of Stein's method. Loosely speaking, to use the framework one begins with a generator of a CTMC whose stationary distribution one wishes to approximate. By performing Taylor expansion on this generator, which contains some sort of difference equations, one extracts a second order differential operator that becomes the generator of the diffusion approximation. The rest of the terms in the expansion become the approximation error. Previous diffusion approximations that had constant diffusion coefficients captured all the first order terms, but only part of the second order terms in the expansion. In contrast, our diffusion approximation captures the entire first and second order terms in the Taylor expansion of the CTMC generator, thus making the error smaller. The method we use to derive our high order approximation is not unique to the $M/M/n$ system, and is applicable to other systems that are amenable to analysis using the Stein framework. 

In the presence of multiple possible approximations for a system, one must choose which one to use. It was suggested in Remark 2.2 of \cite{GurvHuanMand2014} that a diffusion approximation with a constant diffusion coefficient is of the same quality as one with a state-dependent diffusion coefficient. However, the results in our paper contradict this. When applying Stein's method, we intuitively expect that the best diffusion approximation is one that captures the entire second order Taylor expansion of the CTMC generator. This also suggests that the method can be used as a practical engineering tool to quickly find good approximations that can be used even when error bounds cannot be rigorously established. 

As an example, in \cite{DaiShi2015c} the authors develop a model to study the patient count at a hospital. They use Stein's method to find a diffusion approximation for the midnight patient count, an important quantity in their model. In that paper, the authors consider two diffusion approximations -- one with a constant diffusion coefficient, and another one with a state-dependent diffusion coefficient. The constant coefficient approximation is excellent for a single cluster of 500 beds (in their model, beds play the role of servers). However, a ward cluster of 50 beds or 15 beds is more common in a hospital. For example, a typical ward has 30-60 beds, and  it is quite common for an intensive care unit ward to have 10-20 beds. As demonstrated in their paper, the constant approximation performs poorly with only $20$ beds. However, the state-dependent diffusion coefficient approximation performs remarkably well numerically, although they did not provide an error bound for this approximation. Our paper provides a potential explanation for the excellent performance of their state-dependent approximation. 

Although $Y(\infty)$ is a better approximation than $Y_0(\infty)$ for the Erlang-C system, we do not claim  that it is the best approximation available in the current literature for performance measures of interest in the Erlang-C system. Indeed, in \cite{JansLeeuZwar2011} the authors develop a ``corrected diffusion approximation" for delay probability $\Prob(X(\infty) \geq n)$, the probability that a customer entering the system has to wait for service. The results of that paper are not based on using a different diffusion process to approximate the Erlang-C system. Rather, they rely on a series representation of the cumulative distribution function of a Poisson random variable that was established in \cite[Theorem 2]{JansLeeuZwar2008a}. See also Theorem 8 of the same paper. The purpose of our paper is not to compete with \cite{JansLeeuZwar2011} for the most accurate approximation of specific quantities in the Erlang-C model. Rather it is to emphasize the accuracy gains one can achieve by using diffusion processes with state-dependent diffusion coefficients over ones with constant diffusion coefficients.

We hope that our result will inspire other researchers to focus on diffusion processes that fully capture the second order Taylor expansion term in their models as well. One important direction for future work is to find alternatives to Stein's method for deriving such high order approximations, because so far the class of problems that can be analyzed by Stein's method remains limited.\ The rest of the paper is structured as follows. In Section~\ref{sec:roadmap} we discuss the Poisson equation, gradient bounds and moment bounds that we need to prove Theorem~\ref{thm:w2}. The proofs of these are quite technical and are provided in the appendix. In Section~\ref{sec:proofW} we prove Theorem~\ref{thm:w2}.



We end the introduction by introducing the following notation. For $a, b \in \R$, we use $a^+, a^-, a \wedge b$, and $a \vee b$ to denote $\max(a,0)$, $\max(-a,0)$,  $\min(a,b)$, and $\max(a, b)$, respectively. For a function $f:\R \to \R$, we write $\norm{f}$ to denote $\sup_{x \in \R} \abs{f(x)}$.


\section{Roadmap for Proofs}
\label{sec:roadmap}
In \cite{BravDaiFeng2015} the authors provide a detailed introduction of the generic procedure of using Stein's method for steady-state diffusion approximations. The procedure involves three components -- the Poisson equation and gradient bounds, generator coupling, and moment bounds. In this section we state these components for the Erlang-C system. 

Recall that $R = \lambda/\mu < n$ is the offered load of the system. For notational convenience, we define $\delta>0$ as
\begin{equation}
  \label{eq:delta}  \delta = \frac{1}{ \sqrt{R}} = \sqrt{\frac{\mu}{\lambda}}.
\end{equation}
\subsection{Poisson Equation}
\label{sec:diffPoisson}
The random variable $Y(\infty)$ in
Theorem~\ref{thm:w2} is well-defined and
its density is given in (\ref{eq:stdden}). It turns out that $Y(\infty)$
has the stationary distribution of a diffusion process $Y=\{Y(t), t\ge
0\}$, which we define shortly. We do not prove this claim in this paper since it is not used anywhere in this paper. Nevertheless, it is helpful to think of $Y(\infty)$ in the context of diffusion processes.  The diffusion process $Y$ is defined by its generator, which is given by
\begin{equation}
  \label{eq:GY}
G_Y f(x) =  b(x)f'(x) +  \frac{1}{2}a(x) f''(x) \text{ \quad for $x \in \R$}, \ f \in C^2(\R),
\end{equation}
where $a(x)$ and $b(x)$ are defined in (\ref{eq:adef}) and \eqref{eq:bdef}, respectively.
Clearly,  $b(0)=0$, and  $b(x)$ is Lipschitz continuous. Indeed,
\begin{displaymath}
\abs{  b(x) -b(y)} \le \mu  \abs{x-y} \quad \text{ for } x, y\in \R.
\end{displaymath}
Furthermore, 
\begin{align*}
a(x) \geq \mu > 0, \quad x \in \R,
\end{align*} 
and $a(x)$ is also Lipschitz continuous.
On the other hand, the random variable $Y_0(\infty)$ has the stationary distribution of the diffusion process with generator 
\begin{align}
G_{Y_0} f(x) = b(x)f'(x) +  \frac{1}{2} a(0) f''(x) \text{ \quad for $x \in \R$}, \ f \in C^2(\R). \label{eq:geny0}
\end{align}
The difference between \eqref{eq:GY} and \eqref{eq:geny0} is that the former process has a non-constant diffusion coefficient, whereas the latter has a constant one. Using a non-constant diffusion coefficient is what makes $Y(\infty)$ a better approximation than $Y_0(\infty)$.

We fix $h(x) \in W_2$ with $h(0) = 0$, and consider the Poisson equation 
\begin{align} \label{eq:poisson}
G_Y f_h(x) = \frac{1}{2} a(x) f_h''(x) + b(x) f_h'(x) = \EE h(Y(\infty)) - h(x), \quad  x\in \R.
\end{align}
We use the Lipschitz property of $h(x)$ to see that
\begin{align*}
\abs{\EE h(Y(\infty))} \leq  \EE \big|Y(\infty)\big| < \infty, 
\end{align*}
where the finiteness of $\EE \big|Y(\infty)\big|$ will be proved in \eqref{eq:fbound7}.
One may verify by differentiation that \eqref{eq:poisson} has a family of solutions of the form
\begin{align} \label{eq:poissonsolution}
f_h(x) = a_1 + \int_{0}^{x} \bigg[ a_2 \frac{1}{p(u)} + \frac{1}{p(u)} \int_{-\infty}^{u} \frac{2}{a(y)} \big( \EE h(Y(\infty)) - h(y)\big) p(y) dy\bigg] du, 
\end{align}
where $a_1, a_2 \in \R$ are arbitrary constants, and
\begin{align}
p(x)=\exp\Big({\int_0^x \frac{2b(y)}{a(y)}dy}\Big), \quad x \in \R. \label{eq:pold}
\end{align}
Taking expected values on both sides of (\ref{eq:poisson}) with respect to $\tilde X(\infty)$ yields
\begin{align}
\big| \EE h(\tilde X(\infty)) - \EE h(Y(\infty)) \big| =&\ \big| \EE G_Y f_h(\tilde X(\infty)) \big|.\label{eq:error1}
\end{align}
We will focus on bounding the right hand side, which we do by comparing the generator $G_Y$ to the generator of $\tilde X$ in the next section.

\subsection{Comparing Generators}
\label{sec:compare}
Consider the normalized CTMC
\begin{align*}
\tilde X = \{ \tilde X(t) = \delta(X(t) - R),\ t \geq 0\}.
\end{align*}
It is then clear that $\tilde X(\infty)= \delta\big(X(\infty)-R\big)$ has the stationary distribution of $\tilde X$. The CTMC $\tilde X$ also has a generator.  For any $k \in \Z_+$, we define $x = x_k = \delta(k - R)$. Then for any function $f:\R\to\R$, the generator of $\tilde X$ is given by  
\begin{align}\label{eq:GX}
G_{\tilde X} f(x) = \lambda (f(x + \delta) - f(x)) + d(k) (f(x-\delta) - f(x)),
\end{align}
where
\begin{align} \label{eq:deathrate}
d(k) = \mu (k \wedge n),
\end{align}
is the departure rate corresponding to the system having $k$ customers.  One may check that 
\begin{align*}
b(x) = \delta( \lambda - d(k)), \quad a(x) = \delta^2 (\lambda + d(k) 1(k > 0)) .
\end{align*}
The relationship between $G_{\tilde X}$ and the stationary distribution of $\tilde X$ is illustrated by the following lemma. For a proof of this lemma, we refer the reader to Lemma~1 of \cite{BravDaiFeng2015}, whose statement and proof are nearly identical. In \cite{BravDaiFeng2015}, the authors used a cubic Lyapunov function to show that $\EE \big( X(\infty)\big)^2 < \infty$. To prove Lemma~\ref{lem:gz}, we would repeat their arguments with a quartic Lyapunov function to show that $\EE \big( X(\infty)\big)^3 < \infty$.
\begin{lemma} \label{lem:gz}
Let $f(x): \R \to \R$ be a function such that $\abs{f(x)} \leq C(1+x)^3$ for some $C > 0$ (i.e.\ $f(x)$ is dominated by a cubic function), and assume that the CTMC $\tilde X$ is positive recurrent. Then
\begin{align*}
\EE \big[ G_{\tilde X} f(\tilde X(\infty)) \big] = 0.
\end{align*}
\end{lemma}
\begin{remark}
As we will see in Lemma~\ref{lem:gb}, the solutions to the Poisson equation \eqref{eq:poisson} have bounded second derivatives, meaning that they satisfy the conditions of Lemma~\ref{lem:gz}.
\end{remark}

Suppose for now that for any $h(x) \in W_2$, the solution to the Poisson equation $f_h(x)$ satisfies the conditions of Lemma~\ref{lem:gz}. Applying Lemma~\ref{lem:gz} to \eqref{eq:error1}, we see that
\begin{align}
\big| \EE h(\tilde X(\infty)) - \EE h(Y(\infty)) \big| =&\ \big| \EE G_Y f_h(\tilde X(\infty)) \big| \notag \\
=&\ \big| \EE G_{\tilde X} f_h(\tilde X(\infty)) - \EE G_Y f_h(\tilde X(\infty)) \big| \notag \\
\leq&\ \EE \big| G_{\tilde X} f_h(\tilde X(\infty)) -  G_Y f_h(\tilde X(\infty)) \big|, \label{eq:gen_bound}
\end{align}
which is the standard generator coupling step when applying Stein's method to steady-state diffusion approximations.
\subsection{Taylor Expansion}
\label{sec:taylor}

To bound the right side of (\ref{eq:gen_bound}), 
we study the difference $G_{\tilde X} f_h(x) - G_Y f_h(x)$. For that we perform
Taylor expansion on $G_{\tilde X}f_h(x)$. It is not hard to deduce from \eqref{eq:poissonsolution}
that $f_h''(x)$ exists for all $x \in \R$ and is absolutely continuous. Furthermore, $f_h'''(x)$ is discontinuous at the points $x = -1/\delta$, and $x = -\zeta$, and therefore we write $f_h'''(x-)$ to denote $\lim_{u \uparrow x}f_h'''(u)$. We first define 
\begin{align}
\epsilon_1(x) =&\ \frac{1}{2} \int_x^{x+\delta} (x+\delta -y)^2(f_h'''(y)-f_h'''(x-))dy, \quad x \in \R, \label{eq:eps1def} \\
\epsilon_2(x) =&\  -\frac{1}{2} \int_{x-\delta}^{x} (y-(x-\delta))^2(f_h'''(y)-f_h'''(x-))dy, \quad x \in \R. \label{eq:eps2def}
\end{align}
Now observe that
\begin{align*}
f_h(x+\delta) -f_h(x) =&\  f_h'(x) \delta  + \int_x^{x+\delta} (f_h'(y)- f_h'(x))dy  \\
 =&\  f_h'(x) \delta  + \int_x^{x+\delta} (x+\delta -y)f_h''(y)dy  \\
=&\  f_h'(x) \delta  +  \frac{1}{2}\delta^2 f_h''(x) +  \int_x
^{x+\delta} (x+\delta -y)(f_h''(y)-f_h''(x))dy  \\
=&\  f_h'(x) \delta  +  \frac{1}{2}\delta^2 f_h''(x) + \frac{1}{6}\delta^3 f_h'''(x-) \\
&+ \frac{1}{2} \int_x
^{x+\delta}  (x+\delta -y)^2(f_h'''(y)-f_h'''(x-))dy  \\
=&\  f_h'(x) \delta  + \frac{1}{2}\delta^2 f_h''(x) + \frac{1}{6}\delta^3 f_h'''(x-)+  \epsilon_1(x),
\end{align*}
and similarly, one can check that
\begin{align*}
 (f_h(x-\delta) -f_h(x))
&= -f_h'(x) \delta  +  \frac{1}{2}\delta^2 f_h''(x)-\frac{1}{6}\delta^3 f_h'''(x-)  + \epsilon_2(x).
\end{align*}
Then for any $k \in \Z_+$, and $x = x_k = \delta(k -
R)$, we recall that $b(x) = \delta(\lambda - d(k))$ to see
that
\begin{align}
  G_{\tilde X}f_h(x) =&\ \lambda  \delta f_h'(x) + \lambda \frac{1}{2}\delta^2 f_h''(x) + \frac{1}{6}\lambda \delta^3 f_h'''(x-) + \lambda \epsilon_1(x) \notag \\
  & - d(k) \delta f_h'(x) + d(k) \frac{1}{2} \delta^2 f_h''(x) -d(k) \frac{1}{6}\delta^3 f_h'''(x-) + d(k)
  \epsilon_2(x) \notag \\
=&\ b(x) f_h'(x) + (\lambda +d(k)1(x \geq -1/\delta)) \frac{1}{2} \delta^2 f_h''(x) \notag \\
& + \frac{1}{6}\delta^3(\lambda - d(k))f_h'''(x-) + \lambda \epsilon_1(x)+ (\lambda-\frac{1}{\delta}b(x) )
  \epsilon_2(x) \notag \\
=&\ G_Y f_h(x) + \frac{1}{6}\delta^2b(x)f_h'''(x-) + \lambda(\epsilon_1(x)+\epsilon_2(x)) - \frac{1}{\delta}b(x)\epsilon_2(x). \label{eq:taylor}
\end{align}
Therefore, 
\begin{align}
\Big| \EE h(\tilde X(\infty)) - \EE h(Y(\infty)) \Big| \leq&\ \frac{1}{6} \delta^2 \EE \Big[ \big|f_h'''(\tilde X(\infty)-)b(\tilde X(\infty)) \big| \Big] + \lambda \EE\Big[ \big| \epsilon_1(\tilde X(\infty))\big|\Big] \notag \\
&+ \lambda \EE\Big[ \big| \epsilon_2(\tilde X(\infty))\big|\Big] +  \frac{1}{\delta}  \EE\Big[ \big|b(\tilde X(\infty)) \epsilon_2(\tilde X(\infty))\big|\Big], \label{eq:first_bounds}
\end{align}
where $\epsilon_1(x)$ and $\epsilon_2(x)$ are as in \eqref{eq:eps1def} and \eqref{eq:eps2def}.

Our approximation $Y(\infty)$ is much better than the one used in \cite{BravDaiFeng2015}, and the reason for this can be seen in the Taylor expansion in \eqref{eq:taylor}. Since we use a non-constant diffusion coefficient $a(x)$, our diffusion approximation is able to capture the entire second order term in the Taylor expansion of $G_{\tilde X} f_h(x)$ (i.e.\ all the terms that correspond to $f_h'(x)$ and $f_h''(x)$). In contrast, the approximation in \cite{BravDaiFeng2015} uses a constant diffusion coefficient $a(0) = 2\mu $, meaning that they have an extra error term of the form 
\begin{align*}
\frac{1}{2}\delta^2\EE \Big[ \big| f_h''(\tilde X(\infty)) \big( a(\tilde X(\infty)) - a(0) \big)   \big| \Big],
\end{align*}
which turns out to be on the order of $\delta$, not $\delta^2$.

 In the following sections we state the gradient bounds and moment bounds required to bound \eqref{eq:first_bounds}, and also describe how to handle the $\epsilon_1(x)$ and $\epsilon_2(x)$ terms there.

\subsection{Moment Bounds and Gradient Bounds}
\label{sec:momentbounds}
We now state several lemmas that establish necessary moment and gradient bounds for
showing that \eqref{eq:first_bounds} is small.  We recall that $\zeta < 0$.
\begin{lemma} \label{lem:moment_bounds}
For all $n \geq 1, \lambda > 0$, and $\mu > 0$ satisfying $0 < R < n $,
\begin{align}
&\EE \Big[(\tilde X(\infty))^2 1(\tilde X(\infty) \leq -\zeta)\Big] \leq \frac{4}{3} + \frac{2\delta^2}{3}, \label{eq:xsquaredelta}\\
&\EE \Big[ \big|\tilde X(\infty) 1(\tilde X(\infty) \leq -\zeta)\big| \Big] \leq \sqrt{\frac{4}{3} + \frac{2\delta^2}{3} }, \label{eq:xminusdelta}\\
&\EE \Big[ \big|\tilde X(\infty) 1(\tilde X(\infty) \leq -\zeta)\big| \Big] \leq 2\abs{\zeta} \label{eq:xminuszeta}\\
&\EE \Big[\big|\tilde X(\infty)1(\tilde X(\infty) \geq -\zeta)\big| \Big] \leq \frac{1}{\abs{\zeta}} + \frac{\delta^2}{4\abs{\zeta}} + \frac{\delta}{2}, \label{eq:xplus}\\
&\Prob(\tilde X(\infty) \leq -\zeta) \leq (2+\delta)\abs{\zeta}. \label{eq:idle_prob}
\end{align}
\end{lemma} 
Lemma~\ref{lem:moment_bounds} is just a restatement of Lemma~2 from \cite{BravDaiFeng2015}. 
The moment bounds above are not sufficient for us, and the following lemma states some extra moment bounds that are not proved in \cite{BravDaiFeng2015}.
\begin{lemma}
\label{lem:xtramom}
 For all $n \geq 1, \lambda > 0$, and $\mu > 0$ satisfying $0 < R < n $,
\begin{align}
\EE \Big[(\tilde X(\infty))^2 1(\tilde X(\infty) \leq -\zeta) \Big] \leq&\  \big( 5 + \delta (1+\delta/2) \big) \zeta^2 + (2+\delta)\abs{\zeta} \label{eq:xsquarezeta} \\
\EE \Big[(\tilde X(\infty))^2 1(\tilde X(\infty) \geq -\zeta) \Big] \leq&\  \delta^2 +8 + \frac{4}{\abs{\zeta}} \Big( \frac{1}{\abs{\zeta}} + \frac{\delta^2}{4\abs{\zeta}} + \frac{\delta}{2}\Big) +\frac{2(2\delta + \delta^3)}{3\abs{\zeta}}. \label{eq:xsquareplus}
\end{align}
Furthermore, let $\{\pi_k\}_{k=0}^{\infty}$ be the distribution of $X(\infty)$. Then
\begin{align}
\pi_0 \leq&\ 4(2+\delta)\delta^2\abs{\zeta}, \quad \text{ when } \abs{\zeta} \leq 1, \label{eq:pi0}
\end{align}
and 
\begin{align}
\pi_n \leq \delta \abs{\zeta}. \label{eq:pin}
\end{align}
\end{lemma}
This lemma is proved in Appendix~\ref{app:moment}. The next lemma combines the moment bounds above with some elementary algebra to bound some terms that will appear frequently in the proof of Theorem~\ref{thm:w2} when we bound \eqref{eq:first_bounds}. The proof is provided in Appendix~\ref{app:moment}.
\begin{lemma}
\label{lem:lastbounds}
For all $n \geq 1, \lambda > 0$, and $\mu > 0$ satisfying $1 \leq R < n $,
\begin{align}
\Big(1+\frac{1}{\abs{\zeta}}\Big)\EE \Big[\big|\tilde X(\infty) 1(\tilde X(\infty) \leq -\zeta) \big| \Big] \leq&\ \sqrt{2} + 2, \label{eq:lb1}\\
\Big(1+\frac{1}{\abs{\zeta}}\Big)\EE \Big[(\tilde X(\infty))^2  1(\tilde X(\infty) \leq -\zeta) \Big] \leq&\ 9,\label{eq:lb2}\\
\abs{\zeta}\Prob(\tilde X(\infty) \geq -\zeta) \leq&\ 2,\label{eq:lb3}\\
\zeta^2\Prob(\tilde X(\infty) \geq -\zeta) \leq&\ 20. \label{eq:lb4}
\end{align}  
\end{lemma}
Next we present the gradient bounds, which are proved in Appendix~\ref{app:gradbounds}.
\begin{lemma} \label{lem:gb}
Recall the Poisson equation \eqref{eq:poisson} and that the family of solutions to this equation is given by \eqref{eq:poissonsolution}. In particular, this family is parametrized by constants $a_1, a_2 \in \R$. For $ h(x) \in W_2$ with $h(0) = 0$, let $f_h(x)$ be a solution to the Poisson equation. Then $f_h(x)$ is twice continuously differentiable, with an absolutely continuous second derivative. Moreover, there exists a constant $C > 0$ independent of $\lambda, n$, and $\mu$, such that for all solutions with $a_2 = 0$ and for all $n \geq 1, \lambda > 0$, and $\mu > 0$ satisfying $1 \leq R < n $,
\begin{align}
\abs{f_h'(x)} 
\leq&\
\begin{cases}
\frac{C}{\mu }\Big(1 +  \frac{1}{\abs{\zeta}}\Big), \quad x \leq -\zeta,\\
\frac{C}{\mu \abs{\zeta}}\Big(x + 1 + \frac{1}{\abs{\zeta}}\Big), \quad x \geq -\zeta,
\end{cases}\label{eq:WCder1} \\
\abs{f_h''(x)} \leq&\ 
\begin{cases}
\frac{C}{\mu }\Big(1 +  \frac{1}{\abs{\zeta}}\Big), \quad x \leq -\zeta,\\
\frac{C}{\mu \abs{\zeta}}, \quad x \geq -\zeta,
\end{cases} \label{eq:WCder2}
\end{align}
and 
\begin{align}
 \abs{f_h'''(x)} \leq&\ 
\begin{cases}
\frac{C}{\mu }\Big(1 + \frac{1}{\abs{\zeta}}\Big), \quad x \leq -\zeta,\\
\frac{C}{\mu }, \quad x > -\zeta,
\end{cases} \label{eq:WCder3}
\end{align}
where $f_h'''(x)$ is interpreted as the left derivative at the points $x = -1/\delta$ and $x = -\zeta$.
\end{lemma}
\begin{remark}
In the proof of this lemma we will see that $f_h'''(x)$ is discontinuous at points $x = -1/\delta$, and $x = -\zeta$, see \eqref{eq:hf3}.
\end{remark}
The gradient bounds in Lemma~\ref{lem:gb} are not sufficient for us. To deal with the terms in \eqref{eq:first_bounds} involving $\epsilon_1(x)$ and $\epsilon_2(x)$, we require the following extra bounds, which are proved in Appendix~\ref{app:eterm}.
\begin{lemma}
\label{lem:eterm}
Fix $h(x) \in W_2$ with $h(0)=0$, and let $f_h(x)$ be a solution to the Poisson equation \eqref{eq:poisson} that satisfies the conditions of Lemma~\ref{lem:gb}. Consider only those $x \in \R$ such that $x = x_k = \delta (k - R)$ for some $k \in \Z_+$. Then there exists a constant $C>0$ independent of $\lambda, n$, and $\mu$, such that for all $n \geq 1, \lambda > 0$, and $\mu > 0$ satisfying $1 \leq R < n$,
\begin{align}
\abs{f_h'''(x-) - f_h'''(y)} \leq&\ \frac{C\delta}{\mu } \bigg[ 1(x \leq -\zeta)(1 + \abs{x})\Big(1 + \frac{1}{\abs{\zeta}}\Big)\notag \\
& \hspace{1cm} + 1(x \geq -\zeta+\delta) (1+ \abs{\zeta})\bigg], \quad  y \in (x-\delta, x) \label{eq:eboundleft}
\end{align}
and 
\begin{align}
&\ \abs{f_h'''(x-) - f_h'''(y)}  \notag \\
\leq&\ \frac{C\delta}{\mu } \bigg[ 1(x \leq -\zeta-\delta)(1 + \abs{x})\Big(1 + \frac{1}{\abs{\zeta}}\Big) + 1(x \geq -\zeta) (1+ \abs{\zeta}) \notag \\
& \hspace{1cm}+ \frac{1}{\delta}\Big(1 + \frac{1}{\abs{\zeta}}\Big)1(x\in \{-1/\delta, -\zeta\}) \bigg], \quad y \in (x, x + \delta). \label{eq:eboundright}
\end{align}
\end{lemma}
\begin{remark}
The upper bound in \eqref{eq:eboundright} has an extra term compared to the bound in \eqref{eq:eboundleft}. This terms is the result of the discontinuity of $f_h'''(x)$ at $x = -1/\delta$ and $x = -\zeta$.
\end{remark}
\section{Proof of Theorem~\ref{thm:w2} }
\label{sec:proofW}
In this section we prove Theorem~\ref{thm:w2}. Fix $h(x) \in {W_2}$ with $h(0) = 0$, and recall that the family of solutions to the Poisson equation is given by \eqref{eq:poissonsolution}. For the remainder of Section~\ref{sec:proofW}, we fix one such solution $f_h(x)$ with $a_2 = 0$. We will now bound \eqref{eq:first_bounds}, which we recall here as
\begin{align}
\Big| \EE h(\tilde X(\infty)) - \EE h(Y(\infty)) \Big| \leq&\ \frac{1}{6} \delta^2 \EE \Big[ \big|f_h'''(\tilde X(\infty)-)b(\tilde X(\infty)) \big| \Big] + \lambda \EE\Big[ \big| \epsilon_1(\tilde X(\infty))\big|\Big] \notag \\
&+ \lambda \EE\Big[ \big| \epsilon_2(\tilde X(\infty))\big|\Big] +  \frac{1}{\delta}  \EE\Big[ \big|b(\tilde X(\infty)) \epsilon_2(\tilde X(\infty))\big|\Big], \label{eq:second_bounds}
\end{align}
where
\begin{align*}
\epsilon_1(x) =&\ \frac{1}{2} \int_x^{x+\delta} (x+\delta -y)^2(f_h'''(y)-f_h'''(x-))dy, \\
\epsilon_2(x) =&\  -\frac{1}{2} \int_{x-\delta}^{x} (y-(x-\delta))^2(f_h'''(y)-f_h'''(x-))dy.
\end{align*}
\begin{proof}[Proof of Theorem~\ref{thm:w2}]
Throughout the proof we assume that $R \geq 1$, or equivalently, $\delta \leq 1$. We will use $C > 0$ to denote a generic constant that may change from line to line, but does not depend of $\lambda,n$, and $\mu$. Suppose we know that for some positive constants $c_1, \ldots, c_4 > 0$ independent of $\lambda, n$, and $\mu$, 
\begin{align}
&\ \Big| \EE h(\tilde X(\infty)) - \EE h(Y(\infty)) \Big| \leq  \delta^2 (c_1 + c_2  +c_3+ \delta c_4)+ C\delta(\pi_0+\pi_n), \label{inl:thm}
\end{align}
where $\{\pi_k\}_{k=0}^{\infty}$ is the distribution of $X(\infty)$. Then to prove the theorem we would only need to show that 
\begin{align*}
 \pi_0, \pi_n \leq C \delta.
\end{align*}
One way to prove this is to appeal to Theorem~3 of \cite{BravDaiFeng2015}, which states that the Kolmogorov distance 
\begin{align*}
d_K(\tilde X(\infty), Y_0(\infty)) \equiv \sup_{a \in \R} \big| \Prob( \tilde X(\infty) \leq a) - \Prob( Y_0(\infty) \leq a) \big| \leq 188 \delta
\end{align*}
for all $n \geq 1$ and $1 \leq R < n$, where $Y_0(\infty)$ is the random variable with density $\eta(x)$ defined in \eqref{eq:eta}. We would then have that 
\begin{align}
\pi_n =&\ \Prob( -\zeta - \delta/2 \leq \tilde X(\infty) \leq  -\zeta + \delta/2) \notag \\
=&\ \Prob( -\zeta - \delta/2 \leq Y_0(\infty) \leq  -\zeta + \delta/2)  \notag \\
&+ \Prob( -\zeta - \delta/2 \leq \tilde X(\infty) \leq  -\zeta + \delta/2) - \Prob( -\zeta - \delta/2 \leq Y_0(\infty) \leq  -\zeta + \delta/2)  \notag  \\
\leq&\ \delta \norm{\eta} + 2d_K(\tilde X(\infty), Y_0(\infty)) \leq \delta C, \label{eq:pikolmbound}
\end{align}
where in the last inequality we apply Lemma~9 of \cite{BravDaiFeng2015}, which states that $\eta(x)$ is always bounded by $\sqrt{2/\pi}$. The same argument can be used to bound $\pi_0$. 

To conclude the theorem it remains to verify \eqref{inl:thm}, which we do by bounding each of the terms on the right side of \eqref{eq:second_bounds} individually. We recall here that the support of $\tilde X(\infty)$ is a $\delta$-spaced grid, and in particular this grid contains the points $-1/\delta$ and $-\zeta$. In the bounds that follow, we will often consider separately the cases where $\tilde X(\infty) \leq -\zeta$, and $\tilde X(\infty) \geq -\zeta+\delta$. We recall that 
\begin{align*}
b(x) = \mu \big[(x+\zeta)^- + \zeta\big],
\end{align*}
and apply the gradient bound \eqref{eq:WCder3} together with \eqref{eq:lb1} and \eqref{eq:lb3} of Lemma~\ref{lem:lastbounds} to see that
\begin{align*}
\EE \Big[ \big|f_h'''(\tilde X(\infty)-)b(\tilde X(\infty)) \big| \Big] \leq &\ C\Big(1 + \frac{1}{\abs{\zeta}}\Big)\EE \Big[\big|\tilde X(\infty) 1(\tilde X(\infty) \leq -\zeta) \big| \Big] + C\abs{\zeta}\Prob(\tilde X(\infty) \geq -\zeta+\delta) \\
\leq&\ C(\sqrt{2} + 2) + 2C \equiv c_1.
\end{align*}
To bound the next term, we use \eqref{eq:eboundright} from Lemma~\ref{lem:eterm} to see that
\begin{align*}
\lambda \EE\Big[ \big| \epsilon_1(\tilde X(\infty))\big|\Big] \leq&\ \frac{\mu }{2} \EE \bigg[\int_{\tilde X(\infty)}^{\tilde X(\infty)+\delta} \abs{f_h'''(\tilde X(\infty)-)-  f_h'''(y)} dy \bigg]\\
\leq&\ C\delta^2 \EE \bigg[ 1(\tilde X(\infty) \leq -\zeta -\delta)\Big(1 + \big|\tilde X(\infty))\big|\Big)\Big(1 + \frac{1}{\abs{\zeta}}\Big)  \\
&\hspace{1cm}+ 1(\tilde X(\infty) \geq -\zeta) (1+ \abs{\zeta})+ \frac{1}{\delta}\Big(1 + \frac{1}{\abs{\zeta}}\Big)1(\tilde X(\infty)\in \{-1/\delta, -\zeta\})\bigg]\\
\leq&\ C\delta^2 \bigg[\EE \Big[ \big|\tilde X(\infty)1(\tilde X(\infty) \leq -\zeta -\delta) \big|\Big]\Big(1 + \frac{1}{\abs{\zeta}}\Big)  \\
&\hspace{1cm} + \Prob(\tilde X(\infty) \leq -\zeta -\delta)\Big(1 + \frac{1}{\abs{\zeta}}\Big) \\
&\hspace{1cm} + \Prob(\tilde X(\infty) \geq -\zeta) (1+ \abs{\zeta})+\frac{1}{\delta}\Big(1 + \frac{1}{\abs{\zeta}}\Big)(\pi_0 + \pi_n)\bigg],
\end{align*}
where in the last inequality we used the fact that $\Prob(\tilde X(\infty)=-1/\delta)$ and $\Prob(\tilde X(\infty)=-\zeta)$ equal $\pi_0$ and $\pi_n$, respectively. We first use \eqref{eq:idle_prob}, \eqref{eq:lb1}, and \eqref{eq:lb3} to see that
\begin{align*}
\lambda \EE\Big[ \big| \epsilon_1(\tilde X(\infty))\big|\Big] \leq&\ C\delta^2 \Big((\sqrt{2} + 2) + (1 + 3) +(1 + 2)\Big)+ C\delta\Big(1 + \frac{1}{\abs{\zeta}}\Big)(\pi_0+\pi_n) \\ 
\leq&\ C\delta^2 + C\delta(\pi_0+\pi_n) + \frac{C\delta}{\abs{\zeta}} (\pi_0+\pi_n)\\
=&\ C\delta^2 + C\delta(\pi_0+\pi_n) + \frac{C\delta}{\abs{\zeta}} \pi_0 1(\abs{\zeta} \geq 1) + \frac{C\delta}{\abs{\zeta}} \pi_0 1(\abs{\zeta} \leq 1) + \frac{C\delta}{\abs{\zeta}} \pi_n \\
\leq&\ C\delta^2 + C\delta(\pi_0+\pi_n) + \frac{C\delta}{\abs{\zeta}} \pi_0 1(\abs{\zeta} \leq 1) + \frac{C\delta}{\abs{\zeta}} \pi_n.
\end{align*}
Next, we apply the bounds on $\pi_0$ and $\pi_n$ from \eqref{eq:pi0} and \eqref{eq:pin} to conclude that
\begin{align*}
\lambda \EE\Big[ \big| \epsilon_1(\tilde X(\infty))\big|\Big] \leq&\  C\delta^2 + C\delta(\pi_0+\pi_n) \equiv c_2\delta^2 + C\delta(\pi_0+\pi_n).
\end{align*}
We move on to bound the next term in \eqref{eq:second_bounds}. Using \eqref{eq:eboundleft} from Lemma~\ref{lem:eterm},
\begin{align*}
\lambda \EE\Big[ \big| \epsilon_2(\tilde X(\infty))\big|\Big] \leq&\ \frac{\mu }{2} \EE \bigg[\int_{\tilde X(\infty)-\delta}^{\tilde X(\infty)} \abs{f_h'''(\tilde X(\infty)-)-  f_h'''(y)} dy \bigg]\\
\leq&\ C\delta^2 \EE\Big[ 1(\tilde X(\infty) \leq -\zeta)\Big(1 + \big|\tilde X(\infty)\big|\Big)\Big(1+\frac{1}{\abs{\zeta}}\Big) + 1(\tilde X(\infty) \geq -\zeta+\delta) (1+ \abs{\zeta})\Big]\\
\leq&\ C\delta^2 \bigg[\Prob(\tilde X(\infty) \leq -\zeta)\Big(1+\frac{1}{\abs{\zeta}}\Big) + \EE\Big[\big|\tilde X(\infty)1(\tilde X(\infty) \leq -\zeta)\big| \Big]\Big(1+\frac{1}{\abs{\zeta}}\Big) \\
& \hspace{1.5cm}+ \Prob(\tilde X(\infty) \geq -\zeta+\delta) (1+ \abs{\zeta})\bigg].
\end{align*}
Now \eqref{eq:idle_prob}, \eqref{eq:lb1}, and \eqref{eq:lb3} imply that
\begin{align*}
\lambda \EE\Big[ \big| \epsilon_2(\tilde X(\infty))\big|\Big] \leq&\ C\delta^2 \big(4 + (\sqrt{2}+2)+ (1+2)\big) \equiv c_3 \delta^2.
\end{align*}
For the last term in \eqref{eq:second_bounds}, we use the form of $b(x)$ together with \eqref{eq:eboundleft} from Lemma~\ref{lem:eterm} to see that
\begin{align*}
\frac{1}{\delta} \EE\Big[ \big|b(\tilde X(\infty)) \epsilon_2(\tilde X(\infty))\big|\Big] \leq&\ \frac{\delta}{2} \EE \bigg[\big|b(\tilde X(\infty))\big|\int_{\tilde X(\infty)-\delta}^{\tilde X(\infty)} \abs{f_h'''(\tilde X(\infty)-)-  f_h'''(y)} dy \bigg]\\
\leq&\ C\delta^3 \bigg[ \EE \Big[\big|\tilde X(\infty)(1 + \big|\tilde X(\infty)\big|) 1(\tilde X(\infty) \leq -\zeta)\big|\Big] \Big(1+\frac{1}{\abs{\zeta}}\Big) \notag \\
& \hspace{1cm} + \Prob(\tilde X(\infty) \geq -\zeta+\delta)\abs{\zeta} (1+ \abs{\zeta})\bigg]\\
\leq&\ C\delta^3 \bigg[ \EE \Big[\big|\tilde X(\infty) 1(\tilde X(\infty) \leq -\zeta)\big|\Big]\Big(1+\frac{1}{\abs{\zeta}}\Big) \\
&\hspace{1cm}+ \EE \Big[\big(\tilde X(\infty)\big)^21(\tilde X(\infty) \leq -\zeta)\big|\Big] \Big(1+\frac{1}{\abs{\zeta}}\Big) \notag \\
& \hspace{1cm} + \Prob(\tilde X(\infty) \geq -\zeta+\delta)(\abs{\zeta}+\abs{\zeta}^2 )\bigg].
\end{align*}
We apply \eqref{eq:lb1}--\eqref{eq:lb4} from Lemma~\ref{lem:lastbounds} to conclude that
\begin{align*}
\frac{1}{\delta} \EE\Big[ \big|b(\tilde X(\infty)) \epsilon_2(\tilde X(\infty))\big|\Big] \leq&\ C\delta^3 \big( (\sqrt{2}+2) + 9 +2 + 20\big) \equiv c_4\delta^3.
\end{align*}
Therefore, we have shown that for all $R \geq 1$, and $h(x) \in {W_2}$ with $h(0)=0$, \eqref{inl:thm} holds, concluding the proof of Theorem~\ref{thm:w2}.
\end{proof}
%

\appendix

\section{Moment Bounds}
\label{app:moment}
In this section we prove Lemmas~\ref{lem:xtramom} and \ref{lem:lastbounds}.
\begin{proof}[Proof of Lemma~\ref{lem:xtramom}]
We first prove \eqref{eq:xsquarezeta}, or 
\begin{align*}
\EE \Big[(\tilde X(\infty))^2 1(\tilde X(\infty) \leq -\zeta) \Big] \leq  \big( 5 + \delta (1+\delta/2) \big) \zeta^2 + (2+\delta)\abs{\zeta}.
\end{align*}
Let $x$ be of the form $x = \delta(k - R)$, where $k \in \Z_+$. Applying $G_{\tilde X}$ to the function $f(x) = \big[ \delta(k-n)^-\big]^2 =  \big[ (x+\zeta)^-\big]^2$, and observing that $1(k \leq n) = 1(x \leq -\zeta)$,  we get
\begin{align*}
G_{\tilde X} f(x) =&\ \lambda 1(x \leq -\zeta - \delta)\big( 2\delta(x + \zeta) + \delta^2\big) + \mu (k \wedge n)1(x \leq -\zeta)\big( -2\delta(x + \zeta)  + \delta^2\big) \notag \\
=&\ \lambda 1(x \leq -\zeta)\big( 2\delta(x + \zeta) + \delta^2\big) + \mu k 1(x \leq -\zeta)\big( -2\delta(x + \zeta)  + \delta^2\big) - \delta^2\lambda 1(x = -\zeta) \notag \\
=&\ 1(x \leq -\zeta)\Big( 2\delta(x + \zeta)(\lambda - \mu k) + \delta^2(\lambda + \mu k)  \Big) - \delta^2\lambda 1(x = -\zeta)  \notag \\
=&\ 1(x \leq -\zeta)\Big( 2\delta(x + \zeta)(\lambda -  \mu n) + 2\mu \delta(x + \zeta)(n-k) + \delta^2(\lambda +\mu k)  \Big)  - \delta^2\lambda 1(x = -\zeta) \\
=&\ 1(x \leq -\zeta)\Big( 2\mu (x + \zeta)\zeta - 2\mu (x + \zeta)^2 + \delta^2(\lambda + \mu k)  \Big)  - \delta^2\lambda 1(x = -\zeta).
\end{align*}
Taking expected values on both sides and applying Lemma~\ref{lem:gz}, we see that 
\begin{align*}
&\ \EE \big[(\tilde X(\infty)+\zeta)^2 1(\tilde X(\infty) \leq -\zeta) \big]\\
\leq&\ \zeta \EE \big[(\tilde X(\infty)+\zeta) 1(\tilde X(\infty) \leq -\zeta) \big] +  \frac{\delta^2}{2\mu}\Prob(\tilde X(\infty) \leq -\zeta)( \lambda +  \mu n)\\
=&\ \zeta \EE \big[(\tilde X(\infty)+\zeta) 1(\tilde X(\infty) \leq -\zeta) \big] + \frac{1}{2}\Prob(\tilde X(\infty) \leq -\zeta)(1 + \delta^2 n).
\end{align*} 
To proceed, we rely on (B.5) from \cite{BravDaiFeng2015}, which tells us that $\EE \big[(\tilde X(\infty)+\zeta) 1(\tilde X(\infty) \leq -\zeta) \big] = \zeta$, and therefore,
\begin{align*}
\EE \big[(\tilde X(\infty)+\zeta)^2 1(\tilde X(\infty) \leq -\zeta) \big] \leq&\  \zeta^2 + \frac{1}{2}\Prob(\tilde X(\infty) \leq -\zeta)(1 + \delta^2 n)\\
\leq&\ \zeta^2 + \frac{2+\delta}{2} \abs{\zeta}(1 + \delta^2 n),
\end{align*}
where we used \eqref{eq:idle_prob} to get the last inequality. Since $\abs{\zeta} = \delta (n - \frac{\lambda}{\mu })$,
\begin{align*}
\delta^2 n  =  \delta^2 \frac{\abs{\zeta}}{\delta} + \delta^2 \frac{\lambda}{\mu } = \delta \abs{\zeta} + 1,
\end{align*}
and hence,
\begin{align*}
\EE \big[(\tilde X(\infty) + \zeta)^2 1(\tilde X(\infty) \leq -\zeta) \big] \leq&\  \zeta^2 + \frac{2+\delta}{2}\abs{\zeta}(2 + \delta \abs{\zeta}).
\end{align*}
By expanding the square inside the expected value on the left hand side and using \eqref{eq:xminuszeta}, we see that
\begin{align*}
\EE \big[(\tilde X(\infty))^2 1(\tilde X(\infty) \leq -\zeta) \big] \leq&\  \zeta^2 + \frac{2+\delta}{2}\abs{\zeta}(2 + \delta \abs{\zeta}) + 2\abs{\zeta}\EE\Big[ \big| \tilde X(\infty) 1(\tilde X(\infty) \leq -\zeta)\big| \Big]\\
\leq&\ 5\zeta^2 + \frac{2+\delta}{2}\abs{\zeta}(2 + \delta \abs{\zeta}) \\
=&\ \big( 5 + \delta (1+\delta/2) \big) \zeta^2 + (2+\delta)\abs{\zeta}.
\end{align*}
This proves \eqref{eq:xsquarezeta}. Now we prove \eqref{eq:xsquareplus}, or 
\begin{align*}
\EE \Big[(\tilde X(\infty))^2 1(\tilde X(\infty) \geq -\zeta) \Big] \leq \delta^2 +8 + \frac{4}{\abs{\zeta}} \Big( \frac{1}{\abs{\zeta}} + \frac{\delta^2}{4\abs{\zeta}} + \frac{\delta}{2}\Big) +\frac{2(2\delta + \delta^3)}{3\abs{\zeta}}.
\end{align*}
Let $x$ be of the form $x = \delta(k - R)$, where $k \in \Z_+$. 
Recall that 
\begin{align*}
b(x) = \mu \big[\zeta + (x+\zeta)^- \big] = \delta (\lambda - \mu(k \wedge n)).
\end{align*}
Set $a = \delta\big(\lfloor R \rfloor - R\big) < 0$, and consider the function $f(x) = x^3 1( x \geq a+\delta)$. Then 
\begin{align}
G_{\tilde X} f(x) =&\ \lambda 1( x \geq a+\delta) \big((x+\delta)^3 - x^3 \big) +\lambda 1( x = a) (x+\delta)^3  \notag \\
&+ \mu (k \wedge n)1( x > a+\delta) \big((x-\delta)^3 - x^3\big) + \mu (k \wedge n)1( x = a+\delta) ( - x^3) \notag \\
=&\ 1( x \geq a+\delta) \Big[ \lambda  \big((x+\delta)^3 - x^3 \big) + \mu (k \wedge n) \big((x-\delta)^3 - x^3\big)\Big] \notag \\
&+ \lambda 1( x = a) (x+\delta)^3 - \mu (k \wedge n)1( x = a+\delta) (x-\delta)^3. \label{eq:gfcube}
\end{align}
Suppose $x \geq a+\delta$. Using the fact that $1( x = a+\delta) = 1( k = \lfloor R \rfloor + 1)$, we see that
\begin{align}
G_{\tilde X} f(x) =&\ \lambda (3\delta x^2 + 3\delta^2 x + \delta^3) + \mu (k \wedge n) (-3\delta x^2 + 3\delta^2 x - \delta^3) - a^3 \mu (k \wedge n)1( x = a+\delta) \notag \\
=&\ 3\delta x^2(\lambda - \mu(k \wedge n)) + 3\delta^2 x(\lambda + \mu(k \wedge n)) + \delta^3(\lambda - \mu(k \wedge n))  \notag \\ 
& - a^3 \mu \big((\lfloor R \rfloor + 1) \wedge n\big) 1( x = a+\delta) \notag  \\
=&\ 3x^2b(x) + 3\delta^2 x\big(2\lambda - (\lambda - \mu(k \wedge n))\big) + \delta^2b(x) \notag \\
&- a^3 \mu \big((\lfloor R \rfloor + 1) \wedge n\big) 1( x = a+\delta) \notag  \\
=&\ 3x^2b(x) + 6\mu x - 3\delta xb(x) + \delta^2b(x)- a^3 \mu \big((\lfloor R \rfloor + 1) \wedge n\big) 1( x = a+\delta). \label{eq:gtildx}
\end{align}
When $x \in [a+\delta,  -\zeta)$ (which is the empty interval if $\lfloor R \rfloor + 1 = n$), then $b(x) = -\mu x$, and
\begin{align}
G_{\tilde X} f(x) =&\ -3\mu x^3 + 6\mu x + 3\delta \mu x^2 - \delta^2\mu x - a^3 \mu \big((\lfloor R \rfloor + 1) \wedge n\big) 1( x = a+\delta) \notag \\
\leq&\ -3\mu \big(x^3 -\delta x^2 + \frac{1}{3}\delta^2 x \big) + 6\mu x + \delta^3 \mu (\lfloor R \rfloor + 1) \notag  \\
\leq&\ -3\mu \big(x^3 -\delta x^2 + \frac{1}{3}\delta^2 x \big) + 6\mu x + \delta \mu  + \delta^3 \mu  \notag \\
\leq&\ 6\mu x+ \delta \mu  + \delta^3 \mu, \label{eq:gtildx1}
\end{align}
where in the first inequality we used the fact that $\abs{a} \leq \delta$, and in the last inequality we used the fact that $g(x) \equiv x^3 -\delta x^2 + \frac{1}{3}\delta^2 x \geq 0$ for all $x \geq 0$, which is true because $g(0) = 0$ and $g'(x) \geq 0$ for all $x \in \R$. Now when $x \geq -\zeta$, then $b(x) = -\mu \abs{\zeta}$, and using \eqref{eq:gtildx} we see that
\begin{align}
G_{\tilde X} f(x) =&\ -3x^2\mu \abs{\zeta} + 6\mu x + 3\delta x\mu \abs{\zeta} - \delta^2\mu \abs{\zeta} - a^3 \mu \big((\lfloor R \rfloor + 1) \wedge n\big) 1( x = a+\delta) \notag \\
\leq&\ -3\mu \abs{\zeta}\big(x^2-\delta x\big) + 6\mu x - a^3 \mu \big((\lfloor R \rfloor + 1) \wedge n\big) 1( x = a+\delta)  \notag \\
\leq&\ -3\mu \abs{\zeta}\big(x^2-\delta x\big) + 6\mu x + \delta \mu  + \delta^3 \mu  \notag \\
\leq&\ -3\mu \abs{\zeta}\big(\frac{1}{2}x^2-\frac{1}{2}\delta^2\big) + 6\mu x + \delta \mu  + \delta^3 \mu , \label{eq:gtildx2}
\end{align}
Combining \eqref{eq:gtildx1} and \eqref{eq:gtildx2} with \eqref{eq:gfcube}, we have just shown that 
\begin{align*}
G_{\tilde X} f(x) \leq&\ -\frac{3}{2}\mu \abs{\zeta}\big(x^2-\delta^2\big)1(x \geq -\zeta)+ 6\mu x 1( x \geq a+\delta) \\
&+ \delta \mu  + \delta^3 \mu + \lambda 1( x = a) (x+\delta)^3.
\end{align*}
Taking expected values on both sides above, and applying Lemma~\ref{lem:gz}, we see that 
\begin{align*}
\frac{3}{2}\mu \abs{\zeta} \EE\Big[ (\tilde X(\infty))^2 1(\tilde X(\infty) \geq -\zeta) \Big] \leq&\ \frac{3}{2}\mu \abs{\zeta}\delta^2 + 6\mu \EE \Big[\big| \tilde X(\infty) \big| \Big]+ \delta \mu  + \delta^3 \mu + \lambda (a+\delta)^3,
\end{align*}
and since $\lambda (a+\delta)^3 \leq \lambda \delta^3 = \mu \delta$, we have
\begin{align*}
\EE\Big[ (\tilde X(\infty))^2 1(\tilde X(\infty) \geq -\zeta) \Big] \leq&\ \delta^2 + \frac{4}{\abs{\zeta}} \EE \Big[\big| \tilde X(\infty) \big| \Big] +\frac{2(2\delta + \delta^3)}{3\abs{\zeta}}.
\end{align*}
Using the moment bounds in \eqref{eq:xminuszeta} and \eqref{eq:xplus}, we conclude that 
\begin{align*}
\EE\Big[ (\tilde X(\infty))^2 1(\tilde X(\infty) \geq -\zeta) \Big] \leq&\ \delta^2 +8 + \frac{4}{\abs{\zeta}} \Big( \frac{1}{\abs{\zeta}} + \frac{\delta^2}{4\abs{\zeta}} + \frac{\delta}{2}\Big) +\frac{2(2\delta + \delta^3)}{3\abs{\zeta}},
\end{align*}
which proves \eqref{eq:xsquareplus}.

We now prove \eqref{eq:pi0}. Recall that $\{\pi_k\}_{k=0}^{\infty}$ is the distribution of $X(\infty)$. From \eqref{eq:idle_prob} we know that 
\begin{align*}
\Prob(\tilde X(\infty) \leq -\zeta) = \Prob(X(\infty) \leq n) = \sum_{k=0}^{n} \pi_k  \leq (2+\delta)\abs{\zeta}.
\end{align*}
From the flow balance equations, one can see that $\pi_{\lfloor R \rfloor}$ maximizes $\{\pi_k\}_{k=0}^{\infty}$. Now when  $\abs{\zeta} \leq 1$,
\begin{align*}
\abs{\zeta} = \delta (n - R) = \frac{1}{\sqrt{R}}(n - R) \leq 1,
\end{align*}
which implies that 
\begin{align*}
R \geq n - \sqrt{R} \geq n - \sqrt{n},
\end{align*}
where in the last inequality we used $R < n$. We use this inequality together with the fact that $\pi_0 \leq \pi_1 \leq \ldots \leq \pi_{\lfloor R \rfloor}$, which can be verified from the flow balance equations, to see that
\begin{align*}
(2+\delta)\abs{\zeta} \geq \sum_{k=0}^{n} \pi_k \geq \sum_{k=0}^{\lfloor R \rfloor} \pi_k \geq \pi_{0}\lfloor R \rfloor \geq \pi_{0}\lfloor n-\sqrt{n} \rfloor.
\end{align*}
Hence, 
\begin{align*}
\pi_0 \leq \frac{n}{\lfloor n-\sqrt{n} \rfloor} \frac{(2+\delta)\abs{\zeta}}{n} \leq \frac{n}{n-\sqrt{n} - 1} \frac{(2+\delta)\abs{\zeta}}{R} \leq  \frac{4(2+\delta)\abs{\zeta}}{R} = 4(2+\delta)\delta^2 \abs{\zeta}, \quad n \geq 4.
\end{align*}
To conclude the proof of\eqref{eq:pi0} we need to verify the bound above holds for $n < 4$, but this is simple to do. Observe that
\begin{align*}
\pi_0 \leq P(X(\infty) \leq n) \leq (2+\delta)\abs{\zeta}  = (2+\delta)R\delta^2\abs{\zeta}  \leq (2+\delta)n \delta^2 \abs{\zeta} \leq 4(2+\delta)\delta^2\abs{\zeta}, \quad n < 4.
\end{align*}
This proves \eqref{eq:pi0}, and we move on to prove \eqref{eq:pin}. From the flow balance equations corresponding to the CTMC $X$, it is easy to see that 
\begin{align*}
\pi_n = \frac{\frac{R^n}{n!}}{\sum_{k=0}^{n-1} \frac{R^k}{k!} + \frac{R^n}{n!} \frac{1}{1-R/n}} \leq 1 - \frac{R}{n} = \frac{n-R}{n} = \frac{R}{n} \frac{n-R}{R} = \frac{R}{n}\delta \abs{\zeta} \leq \delta \abs{\zeta}.
\end{align*}
This concludes the proof of the lemma.
\end{proof}

\begin{proof}[Proof of Lemma~\ref{lem:lastbounds}]
We recall our assumption that $\delta \leq 1$. We begin by proving \eqref{eq:lb1}. Using the moment bounds in \eqref{eq:xminusdelta} and \eqref{eq:xminuszeta}, we see that
\begin{align*}
(1 + 1/\abs{\zeta})\EE \Big[\big|\tilde X(\infty)  1(\tilde X(\infty) \leq -\zeta) \big|\Big]\leq &\ (1 + 1/\abs{\zeta})\bigg(2\abs{\zeta} \wedge \sqrt{\frac{4}{3} + \frac{2\delta^2}{3} }\bigg) \\
\leq &\ \Big(\sqrt{\frac{4}{3} + \frac{2\delta^2}{3} } + 2\Big) \leq \big(\sqrt{2} + 2\big).
\end{align*}
Next we prove \eqref{eq:lb2}. Using the moment bounds in \eqref{eq:xsquaredelta} and \eqref{eq:xsquarezeta}, we see that
\begin{align*}
&\ (1 + 1/\abs{\zeta})\EE \Big[(\tilde X(\infty))^2  1(\tilde X(\infty) \leq -\zeta) \Big]\\
\leq &\ (1 + 1/\abs{\zeta})\bigg(\Big(\big( 5 + \delta (1+\delta/2) \big) \zeta^2 + (2+\delta)\abs{\zeta} \Big)\wedge \Big(\frac{4}{3} + \frac{2\delta^2}{3}\Big)\bigg) \\
\leq &\ 2 + \bigg(\big(6.5\abs{\zeta} +  3\big) \wedge \frac{2}{\abs{\zeta}}\bigg) \leq 2 + 7,
\end{align*}
where to get the last inequality we considered separately the cases when $\abs{\zeta} \leq 1/2$ and $\abs{\zeta} \geq 1/2$. To prove \eqref{eq:lb3}, we use the moment bound \eqref{eq:xplus} to get
\begin{align*}
\abs{\zeta} \Prob( \tilde X(\infty) \geq -\zeta) \leq&\ \abs{\zeta} \wedge \EE \Big[\big| \tilde X(\infty) 1(\tilde X(\infty) \geq -\zeta)\big| \Big]\notag \\
 \leq&\ \abs{\zeta} \wedge \Big(\frac{1}{\abs{\zeta}} + \frac{\delta^2}{4\abs{\zeta}} + \frac{\delta}{2} \Big) \notag \\
\leq&\ 2, 
\end{align*}
where to get the last inequality we considered separately the cases where $\abs{\zeta} \leq 1$ and $\abs{\zeta} \geq 1$. The proof of \eqref{eq:lb4} is similar. We use the moment bound \eqref{eq:xsquareplus} to see that 
\begin{align*}
\zeta^2 \Prob( \tilde X(\infty) \geq -\zeta) \leq&\ \zeta^2 \wedge \EE \Big[( \tilde X(\infty))^2 1(\tilde X(\infty) \geq -\zeta) \Big]\\
 \leq&\ \zeta^2 \wedge \Big(\delta^2 +8 + \frac{4}{\abs{\zeta}} \Big( \frac{1}{\abs{\zeta}} + \frac{\delta^2}{4\abs{\zeta}} + \frac{\delta}{2}\Big) +\frac{2(2\delta + \delta^3)}{3\abs{\zeta}} \Big)  \\
 \leq&\ \zeta^2 \wedge \Big(1 +8 + \frac{4}{\abs{\zeta}} \Big( \frac{1}{\abs{\zeta}} + \frac{1}{4\abs{\zeta}} + \frac{1}{2}\Big) +\frac{2}{\abs{\zeta}} \Big)  \\
\leq&\ 20,
\end{align*}
where to get the last inequality we considered separately the cases where $\abs{\zeta} \leq 1$ and $\abs{\zeta} \geq 1$.  This concludes the proof of Lemma~\ref{lem:lastbounds}.
\end{proof}

\section{Gradient Bounds}
\label{app:gradbounds}
In this section we prove the gradient bounds in Lemma~\ref{lem:gb}. The proofs of these lemmas follow the arguments used to prove \cite[Lemma 4]{BravDaiFeng2015}, except that here we use a state-dependent diffusion coefficient, whereas \cite{BravDaiFeng2015} have a constant diffusion coefficient. 
For $h(x) \in W_2$, we rewrite the Poisson equation \eqref{eq:poisson} in the more convenient form
\begin{align}
f_h''(x) + \frac{2b(x)}{a(x)}f_h'(x) = \frac{2}{a(x)}\big( \EE h(Y(\infty)) - h(x)\big), \quad x \in \R. \label{eq:sapoisson}
\end{align}
We now introduce some notation related to the Poisson equation and its solutions. Recalling the forms of $a(x)$ and $b(x)$ in \eqref{eq:adef} and \eqref{eq:bdef}, we define the ratio $r(x): \R \to \R$ as
\begin{align*}
r(x)= \frac{2b(x)}{a(x)} = \frac{2\big(\zeta + (x + \zeta)^-\big)}{1 + 1(x > -1/\delta)(1 - \delta(\zeta + (x + \zeta)^-))}, \quad x \in \R.
\end{align*}
Then
\begin{align}
r(x) = 
\begin{cases}
-2x, \quad x \leq -1/\delta,\\
\frac{-2x}{2 +\delta x}, \quad x \in [-1/\delta, -\zeta], \\
\frac{2\zeta}{2 + \delta\abs{\zeta}}, \quad x \geq -\zeta,
\end{cases}
\quad
r'(x) = 
\begin{cases}
-2, \quad x \leq -1/\delta,\\
\frac{-4}{(2+\delta x)^2}, \quad x \in (-1/\delta, -\zeta], \\
0, \quad x > -\zeta,
\end{cases} \label{eq:rform}
\end{align}
where $r'(x)$ is interpreted as the left derivative at the points $x = -1/\delta$ and $x = -\zeta$.
Recall that we defined $p(x)$ in \eqref{eq:pold}. We can see that it satisfies 
\begin{align}
p(x) = 
\begin{cases}
\exp\big(\frac{1}{\delta^2}+ \frac{2}{\delta^2} - \frac{4}{\delta^2} \log(2)\big)\exp(-x^2), \quad x \leq -1/\delta,\\
\exp\big(- \frac{4}{\delta^2} \log(2)\big)\exp\Big[\frac{4}{\delta^2}\log(2 + \delta x) - \frac{2\delta x}{\delta^2}\Big], \quad x \in [-1/\delta, -\zeta], \\
\exp\big(- \frac{4}{\delta^2} \log(2) + \frac{2}{\delta^2}(2\log(2 + \delta \abs{\zeta}) - \delta \abs{\zeta} ) + \frac{2\zeta^2}{2 + \delta \abs{\zeta}} \big)\exp\big(\frac{-2\abs{\zeta} x }{2+\delta \abs{\zeta}}\big), \quad x \geq -\zeta.
\end{cases}\label{eq:pdef}
\end{align}
The following two lemmas will be key to proving Lemma~\ref{lem:gb}. They are proved at the end of this section.

\begin{lemma}\label{lem:handle}
Fix $ h(x) \in W_2$ with $h(0) = 0$, and recall that the family of solutions to the Poisson equation \eqref{eq:sapoisson} is given in \eqref{eq:poissonsolution}. Choose one solution  $f_h(x)$ with parameter $a_2 = 0$. Then
\begin{align}
f_h'(x) \leq&\ \frac{1}{p(x)} \int_{-\infty}^{x} \frac{2}{a(y)}(\EE \big|Y(\infty)\big|  + \abs{y} \big) p(y) dy, \label{eq:hf11} \\
f_h'(x) \leq&\ \frac{1}{p(x)} \int_{x}^{\infty} \frac{2}{a(y)}( \EE \big|Y(\infty)\big|  + \abs{y}\big) p(y) dy, \label{eq:hf12} \\
f_h'''(x) =&\ -r'(x) f_h'(x) - r(x) f_h''(x) - \frac{2}{a(x)}h'(x) - \frac{2a'(x)}{a^2(x)}\big( \EE h(Y(\infty)) - h(x)\big). \label{eq:hf3}
\end{align}
Furthermore, if we assume that 
\begin{align}
\lim_{x \to -\infty}  p(x) f_h''(x) = 0, \label{eq:mlimits}
\end{align}
then 
\begin{align}
f_h''(x) =&\ \frac{1}{p(x)} \int_{-\infty}^{x} \Big(\frac{2}{a(y)} h'(y) - \frac{2a'(y)}{a^2(y)}\big[\EE h(Y(\infty)) - h(y) \big] - r'(y) f_h'(y)\Big)p(y) dy, \label{eq:hf21}
\end{align}
and similarly, if 
\begin{align}
\lim_{x \to \infty}  p(x) f_h''(x) = 0, \label{eq:plimits}
\end{align}
then 
\begin{align}
f_h''(x) =&\  -\frac{1}{p(x)} \int_{x}^{\infty} \Big(\frac{2}{a(y)} h'(y) - \frac{2a'(y)}{a^2(y)}\big[\EE h(Y(\infty)) - h(y) \big] - r'(y) f_h'(y)\Big)p(y) dy. \label{eq:hf22}
\end{align}

\end{lemma}

\begin{lemma} \label{lem:lowlevelbounds}
\allowdisplaybreaks
Recall that $a(x)$, $r(x)$, and $p(x)$  are given by  \eqref{eq:adef}, \eqref{eq:rform}, and \eqref{eq:pdef}, respectively. The following bounds hold:
\begin{align}
&\frac{1}{p(x)} \int_{-\infty}^{x} \frac{2}{a(y)}p(y) dy \leq 
\begin{cases}
\frac{\sqrt{\pi}}{\mu }, \quad x \leq -1/\delta, \label{eq:fbound1} \\
\frac{\sqrt{\pi}}{\mu } + \frac{\sqrt{2\pi}}{\mu }, \quad x \in [-1/\delta, 0],\\
\frac{1}{\mu }e^{\frac{\zeta^2}{2}}(\sqrt{\pi} + \sqrt{2\pi} + \abs{\zeta}), \quad x \in [0,-\zeta],
\end{cases}\\
&\frac{1}{p(x)} \int_{x}^{\infty} \frac{2}{a(y)}p(y) dy \leq 
\begin{cases}
\frac{1}{\mu } \Big( \frac{1}{\abs{\zeta}} + e^{\delta/12} + 1 \Big), \quad x \in [0,-\zeta], \\
\frac{1}{\mu \abs{\zeta}}, \quad x \geq -\zeta,
\end{cases} \label{eq:fbound2} \\
&\frac{1}{p(x)} \int_{-\infty}^{x} \frac{2\abs{y}}{a(y)}p(y) dy \leq 
\begin{cases}
\frac{1}{\mu }, \quad x \leq 0,\\
\frac{2}{\mu } e^{\frac{\zeta^2}{2}}, \quad x \in [0,-\zeta],
\end{cases} \label{eq:fbound3} \\
&\frac{1}{p(x)} \int_{x}^{\infty} \frac{2\abs{y}}{a(y)}p(y) dy \leq 
\begin{cases}
\frac{1}{\mu} + \frac{1}{\mu \zeta^2} + \frac{\delta}{2\mu \abs{\zeta}}, \quad x \in [0,-\zeta],\\
\frac{x}{\mu \abs{\zeta} } + \frac{1}{\mu \zeta^2} + \frac{\delta}{2\mu \abs{\zeta}}, \quad x \geq -\zeta,
\end{cases} \label{eq:fbound4} \\
&\frac{\abs{r(x)}}{p(x)}\int_{-\infty}^{x}\frac{2}{a(y)} p(y)dy \leq 2/\mu, \quad x \leq 0, \label{eq:fbound5} \\
& \frac{\abs{r(x)}}{p(x)}\int_{x}^{\infty}\frac{2}{a(y)} p(y)dy \leq 2/\mu, \quad x \geq 0, \label{eq:fbound6}\\
&\EE \big|Y(\infty)\big| \leq \sqrt{\delta^2 + 2} + \sqrt{2\delta^2 + 4} + \frac{2+\delta^2}{\abs{\zeta}} + \delta. \label{eq:fbound7}
\end{align}
\end{lemma}
With these two lemmas, we are ready to prove Lemma~\ref{lem:gb}.

\subsection{Proof of Lemma~\ref{lem:gb} }
\begin{proof}[Proof of Lemma~\ref{lem:gb} ]
Recall our assumption that $R \geq 1$, or equivalently, $\delta \leq 1$. Throughout the proof we use $C > 0$ to denote a generic constant that does not depend on $\lambda,n$, and $\mu$, and may change from line to line. We begin by bounding $\abs{f_h'(x)}$. We apply \eqref{eq:fbound1}, \eqref{eq:fbound3}, and \eqref{eq:fbound7} to \eqref{eq:hf11} when $x \leq -\zeta$ to see that
\begin{align*}
\mu \abs{f_h'(x)} \leq&\ C\Big( 1 +\frac{1}{\abs{\zeta}} \Big), \quad x \leq 0,\notag \\
\mu \abs{f_h'(x)} \leq&\  2e^{\frac{1}{2}\zeta^2} + e^{\frac{1}{2}\zeta^2}(\sqrt{\pi} + \sqrt{2\pi} + \abs{\zeta})\EE \big|Y(\infty)\big|, \quad x \in [0,-\zeta],
\end{align*}
and apply \eqref{eq:fbound2}, \eqref{eq:fbound4}, and \eqref{eq:fbound7} to \eqref{eq:hf12} when $x \geq 0$ to see that
\begin{align*}
\mu \abs{f_h'(x)} \leq&\ 1 + \frac{1}{\zeta^2} + \frac{\delta}{2\abs{\zeta}} + \Big( \frac{1}{\abs{\zeta}} + e^{\delta/12} + 1 \Big)\EE \big|Y(\infty)\big|, \quad x \in [0,-\zeta], \notag \\
\mu \abs{f_h'(x)} \leq&\ \frac{C}{\abs{\zeta}} \Big(x + 1 + \frac{1}{\abs{\zeta}}\Big), \quad x \geq -\zeta. 
\end{align*}
Above, we see that there are two possible bounds on $\mu \abs{f_h'(x)}$ when $x \in [0, -\zeta]$. By considering separately the cases when $\abs{\zeta} \leq 1$ and $\abs{\zeta} \geq 1$, and using \eqref{eq:fbound7} to bound $\EE \big| Y(\infty) \big|$, we conclude that
\begin{align*}
\mu \abs{f_h'(x)} \leq&\ C\Big( 1 +\frac{1}{\abs{\zeta}} \Big), \quad x \in [0,-\zeta].
\end{align*}
Therefore,
\begin{align}
\abs{f_h'(x)} 
\leq
\begin{cases}
\frac{C}{\mu }\Big( 1 +\frac{1}{\abs{\zeta}} \Big), \quad x \leq -\zeta,\\
\frac{C}{\mu \abs{\zeta}}\Big(x + 1 + \frac{1}{\abs{\zeta}}\Big), \quad x \geq -\zeta,
\end{cases} \label{eq:inlineder1}
\end{align}
which proves \eqref{eq:WCder1}. We now verify \eqref{eq:mlimits} and \eqref{eq:plimits}, which will allow us to use the two forms of $f_h''(x)$ from \eqref{eq:hf21} and \eqref{eq:hf22}. By rearranging the Poisson equation \eqref{eq:sapoisson}, we see that 
\begin{align*}
p(x)f_h''(x) = -p(x) r(x) f_h'(x) + p(x) \frac{2}{a(x)}\big(\EE h(Y(\infty)) - h(x) \big).
\end{align*}
Now the bound on $\abs{f_h'(x)}$ in \eqref{eq:WCder1}, the form of $r(x)$ in  \eqref{eq:rform} and the fact that $h(x) \in W_2$ imply that $f_h'(x)$, $r(x)$ and $h(x)$ grow at most linearly as $x \to \pm \infty$. Furthermore, the form of $a(x)$ in \eqref{eq:adef} suggests that $1/a(x) \leq 1/\mu$ for all $x \in \R$. However, from the form of $p(x)$ in \eqref{eq:pdef} we see that $p(x)$ decreases exponentially fast as $x \to \infty$, and even faster when $x \to -\infty$. This verifies \eqref{eq:mlimits} and \eqref{eq:plimits}. We now bound $\abs{f_h''(x)}$. Since $h(0) = 0$ and $h(x) \in W_2$, it follows that $\abs{h(x)} \leq \abs{x}$ for all $x \in \R$ and $\norm{h'} \leq 1$. From the expressions for $f_h''(x)$ in \eqref{eq:hf21} and \eqref{eq:hf22}, it follows that
\begin{align}
\abs{f_h''(x)} \leq&\ \frac{1}{p(x)} \int_{-\infty}^{x} \Big(\frac{2}{a(y)} +  \frac{2\abs{a'(y)y}}{a^2(y)} + \frac{2\abs{a'(y)}}{a^2(y)}\EE \big| Y(\infty)\big| + \abs{r'(y) f_h'(y)}\Big)p(y) dy, \label{eq:der2bound1} \\
\abs{f_h''(x)} \leq&\ \frac{1}{p(x)} \int_{x}^{\infty} \Big(\frac{2}{a(y)} +  \frac{2\abs{a'(y)y}}{a^2(y)} + \frac{2\abs{a'(y)}}{a^2(y)}\EE \big| Y(\infty)\big| + \abs{r'(y) f_h'(y)}\Big)p(y) dy. \label{eq:der2bound2}
\end{align}
We now bound the terms inside the integrals above. By definition of $a(x)$ in \eqref{eq:adef}, we see that
\begin{align}
a'(x) = \mu \delta 1(x \in (-1/\delta, -\zeta]), \label{eq:ap}
\end{align}
where $a'(x)$ is interpreted as the left derivative for $x = -1/\delta$ and $x = -\zeta$. Therefore,
\begin{align}
\frac{\abs{a'(x)x}}{a(x)} =&\ \frac{\mu \delta \abs{x}}{\mu (2 + \delta x)}1(x \in (-1/\delta, -\zeta]) \leq 1(x \in (-1/\delta, -\zeta]), \label{eq:ap1} \\
\EE \big| Y(\infty) \big| \frac{\abs{a'(x)}}{a(x)} =&\ \EE \big| Y(\infty) \big|\frac{\mu \delta}{\mu (2 + \delta x)}1(x \in (-1/\delta, -\zeta]) \notag \\
 \leq&\ \delta C\Big(1 + \frac{1}{\abs{\zeta}}\Big) 1(x \in (-1/\delta, -\zeta]), \label{eq:ap2}
\end{align}
where in the last inequality we used \eqref{eq:fbound7} and the fact that $\delta \leq 1$ to bound $\EE \big| Y(\infty) \big|$. Furthermore, the form of $r'(x)$ in \eqref{eq:rform} and the bound on $f_h'(x)$ in  \eqref{eq:inlineder1} imply that
\begin{align}
\abs{r'(x)f_h'(x)} =&\ 2\abs{f_h'(x)}1(x \leq  -1/\delta) + \frac{4}{(2+\delta x)^2}\abs{f_h'(x)}1(x \in  (-1/\delta, -\zeta]) \notag \\
\leq&\ 2\abs{f_h'(x)}1(x \leq  -1/\delta) + \frac{4}{2+\delta x}\abs{f_h'(x)}1(x \in  (-1/\delta, -\zeta]) \notag \\
\leq&\  \frac{C}{1 + 1(x \in  (-1/\delta, -\zeta])(1+ \delta x  )} \frac{1}{\mu } \Big(1 +  \frac{1}{\abs{\zeta}}\Big) 1(x \leq -\zeta)\notag \\
=&\  \frac{C}{a(x)} \Big(1 +  \frac{1}{\abs{\zeta}}\Big) 1(x \leq -\zeta). \label{eq:sabound2}
\end{align}
Therefore, when $x \leq -\zeta$ we apply the bounds in \eqref{eq:ap1}--\eqref{eq:sabound2} to \eqref{eq:der2bound1} to see that
\begin{align}
\abs{f_h''(x)} \leq&\ \frac{C}{p(x)} \int_{-\infty}^{x} \frac{1}{a(y)}\Big(1 +  1(y \in (-1/\delta, -\zeta]) + \delta \Big(1 + \frac{1}{\abs{\zeta}}\Big) 1(y \in (-1/\delta, -\zeta]) \notag \\
& \hspace{5cm} + \Big(1 +  \frac{1}{\abs{\zeta}}\Big) 1(y \leq -\zeta)\Big)p(y) dy, \notag \\
\leq&\ \frac{C}{p(x)} \int_{-\infty}^{x} \frac{1}{a(y)}\Big(1 +  \frac{1}{\abs{\zeta}}\Big)p(y) dy, \quad x \leq -\zeta \label{eq:fppboundprocessed1}
\end{align}
and when $x \geq 0$ we apply the same bounds to \eqref{eq:der2bound2} to see that 
\begin{align}
\abs{f_h''(x)} \leq&\ \frac{C}{p(x)} \int_{x}^{\infty} \frac{1}{a(y)}\Big(1 +  1(y \in (-1/\delta, -\zeta]) + \delta \Big(1 + \frac{1}{\abs{\zeta}}\Big) 1(y \in (-1/\delta, -\zeta]) \notag \\
& \hspace{3.5cm} + \Big(1 +  \frac{1}{\abs{\zeta}}\Big) 1(y \leq -\zeta)\Big)p(y) dy \notag  \\
\leq&\ \frac{C}{p(x)} \int_{x}^{\infty} \frac{1}{a(y)}\Big(1 +  \Big(1 +  \frac{1}{\abs{\zeta}}\Big) 1(y \leq -\zeta)\Big)p(y) dy, \quad x \geq 0. \label{eq:fppboundprocessed2}
\end{align} 
We apply \eqref{eq:fbound1} to \eqref{eq:fppboundprocessed1} and \eqref{eq:fbound2} to \eqref{eq:fppboundprocessed2} to get
\begin{align*}
\abs{f_h''(x)} \leq&\ 
\begin{cases}
\frac{C}{\mu }\Big(1 +  \frac{1}{\abs{\zeta}}\Big), \quad x \leq 0,\\
\min \Big\{e^{\zeta^2/2} (\sqrt{\pi} + \sqrt{2\pi} + \abs{\zeta}) , \frac{1}{\abs{\zeta}} + e^{\delta/12} + 1 \Big\}\frac{C}{\mu }\Big(1 +  \frac{1}{\abs{\zeta}}\Big), \quad x \in [0,-\zeta],\\
\frac{C}{\mu \abs{\zeta}}, \quad x \geq -\zeta,
\end{cases} 
\end{align*}
and by considering separately the cases when $\abs{\zeta} \leq 1$ and $\abs{\zeta} \geq 1$, and recalling that $\delta \leq 1$, we conclude that
\begin{align}
\abs{f_h''(x)} \leq&\ 
\begin{cases}
\frac{C}{\mu }\Big(1 +  \frac{1}{\abs{\zeta}}\Big), \quad x \leq -\zeta,\\
\frac{C}{\mu \abs{\zeta}}, \quad x \geq -\zeta,
\end{cases} \label{eq:inlineder2}
\end{align}
which proves \eqref{eq:WCder2}. 

Now we prove \eqref{eq:WCder3}. Recall the form of $f_h'''(x)$ from \eqref{eq:hf3}, which together with the facts that $\abs{h(x)} \leq \abs{x}$ and $\norm{h'} \leq 1$ implies 
\begin{align*}
\abs{f_h'''(x)} \leq&\ \abs{r'(x) f_h'(x)} + \abs{r(x) f_h''(x)} + \frac{2}{a(x)} + \frac{2\abs{a'(x)}}{a^2(x)}\Big( \abs{x} + \EE \big| Y(\infty)\big| \Big), \quad x \in \R,
\end{align*}
where $f_h'''(x)$ is interpreted as the left derivative at the points $x = -1/\delta$ and $x = -\zeta$. We apply the bound on $\abs{r'(x) f_h'(x)}$ from \eqref{eq:sabound2}, the bounds on $\abs{a'(x) x}/a(x)$ and $\EE \big| Y(\infty)\big|\abs{a'(x)}/a(x)$ from \eqref{eq:ap1} and \eqref{eq:ap2}, and the fact that $1/a(x) \leq 1/\mu$ for all $x \in \R$ to see that 
\begin{align*}
\abs{f_h'''(x)} \leq&\ \frac{C}{\mu }\Big( 1 + \frac{1}{\abs{\zeta}} \Big) 1(x \leq -\zeta) + \frac{C}{\mu }1(x > -\zeta)  + \abs{r(x) f_h''(x)}.
\end{align*}
It remains to bound $\abs{r(x) f_h''(x)}$, but this term does not pose much added difficulty. Indeed, one can multiply both sides of \eqref{eq:fppboundprocessed1} and \eqref{eq:fppboundprocessed2} by $\abs{r(x)}$ and invoke \eqref{eq:fbound5} and \eqref{eq:fbound6} of Lemma~\ref{lem:lowlevelbounds} to arrive at 
\begin{align*}
\abs{r(x) f_h''(x)}  \leq&\ 
\begin{cases}
\frac{C}{\mu }\Big(1 +  \frac{1}{\abs{\zeta}}\Big), \quad x \leq -\zeta,\\
\frac{C}{\mu }, \quad x \geq -\zeta.
\end{cases}
\end{align*}
This proves \eqref{eq:WCder3} and concludes the proof of this lemma.
\end{proof}

\subsection{Proof of Lemma~\ref{lem:handle}}
\begin{proof}[Proof of Lemma~\ref{lem:handle}]
We first prove \eqref{eq:hf11}--\eqref{eq:hf3}. Recall that the family of solutions to the Poisson equation is given by \eqref{eq:poissonsolution}, and is parametrized by constants $a_1, a_2 \in \R$. We fix a solution $f_h(x)$ with $a_2 = 0$, and see that for this solution,
\begin{align}
f_h'(x) =\frac{1}{p(x)} \int_{-\infty}^{x} \frac{2}{a(y)}\big(\EE h(Y(\infty)) - h(y) \big) p(y) dy, \quad x \in \R. \label{eq:fprime1}
\end{align}
Recall that $\nu(x)$, the density of $Y(\infty)$ stated in \eqref{eq:stdden}, is proportional to $p(x)/a(x)$, meaning that
\begin{align*}
\int_{-\infty}^{\infty}\frac{2}{a(y)} \big(\EE h(Y(\infty)) - h(y) \big) p(y) dy = 0.
\end{align*}
Therefore, $f_h'(x)$ can be alternatively written as
\begin{align}
f_h'(x) =&\ -\frac{1}{p(x)} \int_{x}^{\infty} \frac{2}{a(y)}\big(\EE h(Y(\infty)) - h(y)\big) p(y) dy, \quad x \in \R. \label{eq:fprime2}
\end{align}
To obtain the derivative bounds stated in Lemma~\ref{lem:gb}, we need to make use of both expressions for $f_h'(x)$ in \eqref{eq:fprime1} and \eqref{eq:fprime2}, depending on the value of $x$. Observe that if $\EE h(Y(\infty))$ were to be replaced by any other constant, then \eqref{eq:fprime2} would not hold. Since $h(0) = 0$ and $h(x) \in W_2$, we know that $\abs{h(x)} \leq \abs{x}$ for all $x \in \R$. Combining this with \eqref{eq:fprime1} and \eqref{eq:fprime2} proves \eqref{eq:hf11} and \eqref{eq:hf12}.
 

Differentiating the Poisson equation \eqref{eq:sapoisson}, we see that 
\begin{align*}
f_h'''(x) = -r'(x) f_h'(x) - r(x) f_h''(x) - \frac{2}{a(x)}h'(x) - \frac{2a'(x)}{a^2(x)}\big( \EE h(Y(\infty)) - h(x)\big),
\end{align*}
where $a'(x)$ is interpreted as the left derivative at the points $x = -1/\delta$ and $x = -\zeta$. This proves \eqref{eq:hf3}, and we now move on to prove \eqref{eq:hf21} and \eqref{eq:hf22}. Multiplying both sides above by $p(x)$, and observing that $r(x) = p'(x)/p(x)$, we see that 
\begin{align*}
\big( p(x) f_h''(x) \big)' =&\ p(x) f_h'''(x) + p'(x) f_h''(x) \\
=&\ \Big(-r'(x) f_h'(x) - \frac{2}{a(x)}h'(x) - \frac{2a'(x)}{a^2(x)}\big( \EE h(Y(\infty))- h(x) \big)\Big)p(x).
\end{align*}
By assumption,  $\lim_{x \to -\infty} p(x)f_h''(x) = 0$, meaning that we can integrate the equation above to get
\begin{align*}
f_h''(x) =&\ \frac{1}{p(x)} \int_{-\infty}^{x} \Big(-\frac{2}{a(y)} h'(y) - \frac{2a'(y)}{a^2(y)}\big[\EE h(Y(\infty)) - h(y) \big] - r'(y) f_h'(y)\Big)p(y) dy.
\end{align*}
Since we also assumed that $\lim_{x \to \infty} p(x)f_h''(x) = 0$, we have
\begin{align*}
f_h''(x) =&\  -\frac{1}{p(x)} \int_{x}^{\infty} \Big(-\frac{2}{a(y)} h'(y) - \frac{2a'(y)}{a^2(y)}\big[\EE h(Y(\infty)) - h(y) \big] - r'(y) f_h'(y)\Big)p(y) dy.
\end{align*}
This concludes the proof of this lemma.
\end{proof}

\subsection{Proof of Lemma~\ref{lem:lowlevelbounds}}
\begin{proof}[Proof of Lemma~\ref{lem:lowlevelbounds}]
To prove this lemma we verify \eqref{eq:fbound1}--\eqref{eq:fbound7} one at a time. During the proof we will often consider separately the cases when $x$ belongs to one of the intervals $(-\infty, -1/\delta]$, $[-1/\delta, 0]$, $[0,-\zeta]$, and $[-\zeta, \infty)$. We first present several useful inequalities. Observe that for any $x>0$ and $c > 0$,
\begin{equation}
\label{eq:usefulbound}
e^{\frac{c}{2}x^2}\int_{x}^{\infty}{ e^{-\frac{c}{2} y^2} dy} \leq \int_{0}^{\infty}{ e^{-\frac{c}{2} y^2} dy} = \sqrt{\frac{\pi}{2c}}.
\end{equation}
One can verify that the left hand side of (\ref{eq:usefulbound}) peaks at $x=0$ by using the bound
\begin{equation} \label{eq:normcdfbound}
\int_{x}^{\infty}{ e^{-\frac{c}{2} u^2} du} \leq \int_{x}^{\infty}{\frac{u}{x}e^{-\frac{c}{2} u^2} du} = \frac{e^{-cx^2/2}}{cx}
\end{equation}
to see that the derivative of the left side of (\ref{eq:usefulbound}) is negative for $x>0$. 

We now prove \eqref{eq:fbound1}. Recall the form of $a(x)$ from \eqref{eq:adef}, and $p(x)$ from \eqref{eq:pdef}. When $x \leq -1/\delta$, we invoke \eqref{eq:usefulbound} to see that
\begin{align*}
\frac{1}{p(x)} \int_{-\infty}^{x} \frac{2}{a(y)}p(y) dy = e^{x^2} \int_{-\infty}^{x} \frac{2}{\mu }e^{-y^2} dy \leq \frac{2}{\mu}\sqrt{\frac{\pi}{4}} = \frac{\sqrt{\pi}}{\mu }.
\end{align*}
Now when $x \in [-1/\delta, 0]$,
\begin{align}
\frac{1}{p(x)} \int_{-\infty}^{x} \frac{2}{a(y)}p(y) dy =&\ e^{\frac{3}{\delta^2}}e^{\frac{2}{\delta^2}[\delta x - 2\log(2 + \delta x)]}\int_{-\infty}^{-1/\delta} \frac{2}{\mu }e^{-y^2} dy \notag \\
&+ e^{\frac{2}{\delta^2}[\delta x - 2\log(2 + \delta x)]}\int_{-1/\delta}^{x} \frac{2}{\mu(2 + \delta y) }e^{\frac{2}{\delta^2}[2\log(2 + \delta y) - \delta y]} dy. \label{eq:sa1}
\end{align}
By differentiating the function $e^{\frac{2}{\delta^2}[\delta x - 2\log(2 + \delta x)]}$, one can see that it achieves its maximum  on the interval $[-1/\delta, 0]$ at the point $x = -1/\delta$, and therefore
\begin{align}
\sup_{x \in (-1/\delta, 0]} e^{\frac{2}{\delta^2}[\delta x - 2\log(2 + \delta x)]} = e^{-2/\delta^2}. \label{eq:logsup}
\end{align} 
We use \eqref{eq:usefulbound} and \eqref{eq:logsup} to see that for $x \in [-1/\delta, 0]$, 
\begin{align}
e^{\frac{3}{\delta^2}}e^{\frac{2}{\delta^2}[\delta x - 2\log(2 + \delta x)]}\int_{-\infty}^{-1/\delta}  \frac{2}{\mu }e^{-y^2} dy \leq&\  \frac{2}{\mu } e^{\frac{3}{\delta^2}} e^{-2/\delta^2} e^{-1/\delta^2} \frac{\sqrt{\pi}}{2} = \frac{\sqrt{\pi}}{\mu }. \label{eq:fonetermone}
\end{align}
 To bound the second term in \eqref{eq:sa1}, we claim that 
\begin{align*}
g_1(x) \equiv e^{\frac{2}{\delta^2}[\delta x-2\log(2 + \delta x)]}\int_{-1/\delta}^{x} \frac{2}{\mu(2 + \delta y) }e^{\frac{2}{\delta^2}[2\log(2 + \delta y) - \delta y]} dy \leq g_1(0), \quad x \in [-1/\delta, 0]. 
\end{align*}
We prove this by showing that $g_1'(x) \geq 0$ for all $x \in [-1/\delta, 0]$. Observe that 
\begin{align*}
g_1'(x) = \frac{2x}{2+\delta x} g_1(x) + \frac{1}{\mu }\frac{2}{2 + \delta x} = \frac{2}{2 + \delta x}(x g_1(x) + \frac{1}{\mu }), \quad x \in (-1/\delta, 0]. 
\end{align*}
Hence, we want to show that
\begin{align*}
x g_1(x) + \frac{1}{\mu } \geq 0, \quad x \in (-1/\delta, 0]. 
\end{align*} 
Fix $x \in (-1/\delta, 0]$ and observe that
\begin{align*}
\abs{x g_1(x)} = &\ -x e^{\frac{2}{\delta^2}[\delta x-2\log(2 + \delta x)]}\int_{-1/\delta}^{x} \frac{2}{\mu(2 + \delta y) }e^{\frac{2}{\delta^2}[2\log(2 + \delta y) - \delta y]} dy\\
\leq&\ -x e^{\frac{2}{\delta^2}[\delta x-2\log(2 + \delta x)]}\int_{-1/\delta}^{x} \frac{-y}{-x} \frac{2}{\mu(2 + \delta y) }e^{\frac{2}{\delta^2}[2\log(2 + \delta y) - \delta y]} dy\\
=&\ e^{\frac{2}{\delta^2}[\delta x-2\log(2 + \delta x)]}\int_{-1/\delta}^{x}  \frac{-2y}{\mu(2 + \delta y) }e^{\frac{2}{\delta^2}[2\log(2 + \delta y) - \delta y]} dy.
\end{align*}
To solve the integral above, we define $u(y) = \frac{2}{\delta^2}\big[ 2\log(2 + \delta y) - \delta y\big]$, and observe that
\begin{align}
\int_{-1/\delta}^{x}  \frac{-2y}{\mu(2 + \delta y) }e^{\frac{2}{\delta^2}[2\log(2 + \delta y) - \delta y]} dy = \frac{1}{\mu }\int_{-1/\delta}^{x}  u'(y) e^{u(y)} dy = \frac{1}{\mu }\Big[e^{u(y)} \Big|_{y= -1/\delta}^{x} \Big], \label{eq:integral}
\end{align}
and therefore
\begin{align*}
\abs{x g_1(x)} \leq&\ \frac{1}{\mu } e^{\frac{2}{\delta^2}[\delta x-2\log(2 + \delta x)]}\Big[e^{\frac{2}{\delta^2}[2\log(2 + \delta y) - \delta y]} \Big|_{-1/\delta}^{x} \Big]\leq \frac{1}{\mu },
\end{align*}
which implies that the second term in \eqref{eq:sa1} is bounded by $g(0)$. We apply this fact to \eqref{eq:sa1} together with the bound in \eqref{eq:fonetermone} to see that for $x \in [-1/\delta, 0]$,
\begin{align*}
&\ \frac{1}{p(x)} \int_{-\infty}^{x} \frac{2}{a(y)}p(y) dy \leq  \frac{\sqrt{\pi}}{\mu } + e^{\frac{2}{\delta^2}[- 2\log(2 )]}\int_{-1/\delta}^{0} \frac{2}{\mu(2 + \delta y) }e^{\frac{2}{\delta^2}[2\log(2 + \delta y) - \delta y]} dy.
\end{align*}
Using Taylor expansion, we see that
\begin{align}
2\log(2 + \delta y) = 2\log(2) + \frac{2}{2}\delta y - \frac{2}{2^2} \frac{(\delta y)^2}{2} + \frac{2}{(2 + \xi(\delta y))^3} \frac{(\delta y)^3}{6}, \label{eq:logtaylor}
\end{align}
where $\xi(\delta y)$ is some point between $\delta y$ and $0$. Therefore,
\begin{align}
&\ e^{\frac{2}{\delta^2}[-2\log(2)]}\int_{-1/\delta}^{0} \frac{2}{\mu(2 + \delta y) }e^{\frac{2}{\delta^2}[2\log(2 + \delta y) - \delta y]} dy \notag \\ 
=&\ \int_{-1/\delta}^{0} \frac{2}{\mu(2 + \delta y) }e^{\frac{2}{\delta^2}\big[ - \frac{2}{2^2} \frac{(\delta y)^2}{2} + \frac{2}{(2 + \xi(\delta y))^3} \frac{(\delta y)^3}{6}\big]} dy \notag \\
\leq&\ \frac{2}{\mu } \int_{-1/\delta}^{0} e^{-\frac{y^2}{2} } dy \leq \frac{2}{\mu} \sqrt{\frac{\pi}{2}}, \label{eq:fonetwo}
\end{align}
which proves \eqref{eq:fbound1} when $x \in [-1/\delta, 0]$. For the final part of \eqref{eq:fbound1}, we fix $x \in [0,-\zeta]$ and observe that 
\begin{align}
\frac{1}{p(x)} \int_{-\infty}^{x} \frac{2}{a(y)}p(y) dy =&\ e^{\frac{3}{\delta^2}}e^{\frac{2}{\delta^2}[\delta x - 2\log(2 + \delta x)]}\int_{-\infty}^{-1/\delta} \frac{2}{\mu }e^{-y^2} dy \notag \\
&+ e^{\frac{2}{\delta^2}[\delta x - 2\log(2 + \delta x)]}\int_{-1/\delta}^{0} \frac{2}{\mu(2 + \delta y) }e^{\frac{2}{\delta^2}[2\log(2 + \delta y) - \delta y]} dy \notag \\ 
&+ e^{\frac{2}{\delta^2}[\delta x - 2\log(2 + \delta x)]}\int_{0}^{x} \frac{2}{\mu(2 + \delta y) }e^{\frac{2}{\delta^2}[2\log(2 + \delta y) - \delta y]} dy. \label{eq:fonethree}
\end{align}
To bound all three terms on the right hand side, we use the Taylor expansion 
\begin{align}
2\log(2 + \delta y) = 2\log(2) + \frac{2}{2}\delta y - \frac{2}{(2 + \xi(\delta y))^2} \frac{(\delta y)^2}{2} , \label{eq:logtaylor2}
\end{align}
where $\xi(\delta y)$ is some point between $0$ and $\delta y$, to see that  for all $ x \in [0,-\zeta]$,
\begin{align}
e^{\frac{2}{\delta^2}[\delta x - 2\log(2 + \delta x)]} = e^{\frac{2}{\delta^2}[- 2\log(2) + \frac{(\delta x)^2}{(2+\xi(\delta x))^2}]} \leq e^{\frac{2}{\delta^2}[- 2\log(2)]}e^{\frac{x^2}{2}}  \leq e^{\frac{2}{\delta^2}[- 2\log(2)]}e^{\frac{\zeta^2}{2}}. \label{eq:logtaylorineq}
\end{align}
We use \eqref{eq:logtaylorineq} to bound each of the terms on the right hand side of \eqref{eq:fonethree}. For the first term, using \eqref{eq:usefulbound} we get
\begin{align*}
e^{\frac{3}{\delta^2}}e^{\frac{2}{\delta^2}[\delta x - 2\log(2 + \delta x)]}\int_{-\infty}^{-1/\delta}  \frac{2}{\mu }e^{-y^2} dy \leq&\ \frac{2}{\mu } e^{\frac{3}{\delta^2}}e^{\frac{2}{\delta^2}[\delta x - 2\log(2 + \delta x)]}\frac{\sqrt{\pi}}{2} e^{-1/\delta^2} \\
\leq&\  \frac{\sqrt{\pi}}{\mu }e^{2/\delta^2}e^{\frac{\zeta^2}{2}}e^{\frac{2}{\delta^2}[- 2\log(2)]} \leq  \frac{\sqrt{\pi}}{\mu }e^{\frac{\zeta^2}{2}}.
\end{align*}
For the second term in \eqref{eq:fonethree}, we use \eqref{eq:fonetwo} and \eqref{eq:logtaylorineq} to see that
\begin{align*}
&\ e^{\frac{2}{\delta^2}[\delta x - 2\log(2 + \delta x)]}\int_{-1/\delta}^{0} \frac{2}{\mu(2 + \delta y) }e^{\frac{2}{\delta^2}[2\log(2 + \delta y) - \delta y]} dy \\
\leq&\ e^{\frac{\zeta^2}{2}}e^{\frac{2}{\delta^2}[- 2\log(2)]}\int_{-1/\delta}^{0} \frac{2}{\mu(2 + \delta y) }e^{\frac{2}{\delta^2}[2\log(2 + \delta y) - \delta y]} dy \leq e^{\frac{\zeta^2}{2}}\frac{\sqrt{2\pi}}{\mu }.
\end{align*}
To bound the third term in \eqref{eq:fonethree}, we use \eqref{eq:logtaylorineq}, and the fact that $e^{\frac{2}{\delta^2}[2\log(2 + \delta y) - \delta y]}\leq e^{\frac{2}{\delta^2}[2\log(2)]}$ for all $y \geq 0$, which can be checked by differentiating the function $e^{\frac{2}{\delta^2}[2\log(2 + \delta y) - \delta y]}$ to see that it peaks at $y = 0$, to see that
\begin{align*}
&\ e^{\frac{2}{\delta^2}[\delta x - 2\log(2 + \delta x)]}\int_{0}^{x} \frac{2}{\mu(2 + \delta y) }e^{\frac{2}{\delta^2}[2\log(2 + \delta y) - \delta y]} dy \\
\leq&\ e^{\frac{\zeta^2}{2}}e^{\frac{2}{\delta^2}[- 2\log(2)]}\int_{0}^{x} \frac{2}{\mu(2 + \delta y) }e^{\frac{2}{\delta^2}[2\log(2 + \delta y) - \delta y]} dy\\
\leq&\ e^{\frac{\zeta^2}{2}}e^{\frac{2}{\delta^2}[- 2\log(2)]}\int_{0}^{x} \frac{1}{\mu} e^{\frac{2}{\delta^2}[2\log(2)]} dy\\
\leq&\ e^{\frac{\zeta^2}{2}}\int_{0}^{\abs{\zeta}} \frac{1}{\mu } dy 
\leq \frac{\abs{\zeta}}{\mu } e^{\frac{\zeta^2}{2}}.
\end{align*}
We combine these bounds with \eqref{eq:fonethree} to conclude that
\begin{align*}
\frac{1}{p(x)} \int_{-\infty}^{x} \frac{2}{a(y)}p(y) dy \leq \frac{1}{\mu }e^{\frac{\zeta^2}{2}}(\sqrt{\pi} + \sqrt{2\pi} + \abs{\zeta}), \quad x \in [0,-\zeta],
\end{align*}
which proves \eqref{eq:fbound1}. We now prove \eqref{eq:fbound2}. Fix $x \in [0,-\zeta]$ and consider 
\begin{align}
&\ \frac{1}{p(x)} \int_{x}^{\infty} \frac{2}{a(y)}p(y) dy \notag \\
=&\ e^{ \frac{2}{\delta^2}(2\log(2 + \delta \abs{\zeta}) - \delta \abs{\zeta} )}e^{\frac{2\zeta^2}{2 + \delta \abs{\zeta}} }e^{\frac{2}{\delta^2}[\delta x - 2\log(2 + \delta x)]}\int_{-\zeta}^{\infty} \frac{2}{\mu ( 2 + \abs{\zeta}\delta)}e^{-\frac{2\abs{\zeta}}{ 2 + \abs{\zeta}\delta}y} dy \notag \\
&+ e^{\frac{2}{\delta^2}[\delta x - 2\log(2 + \delta x)]}\int_{x}^{-\zeta} \frac{2}{\mu(2 + \delta y) }e^{\frac{2}{\delta^2}[2\log(2 + \delta y) - \delta y]} dy. \label{eq:sa2}
\end{align}
We now bound the first term on the right hand side. By differentiating, we can check that the function $e^{\frac{2}{\delta^2}[ 2\log(2 + \delta x) - \delta x]} $ is monotonically decreasing for $x \geq 0$, implying that
\begin{align}
e^{ \frac{2}{\delta^2}(2\log(2 + \delta \abs{\zeta}) - \delta \abs{\zeta} )}e^{\frac{2}{\delta^2}[\delta x - 2\log(2 + \delta x)]} \leq 1, \quad x \geq 0. \label{eq:monotone}
\end{align}
Therefore, the first term on the right side of \eqref{eq:sa2} can be bounded by 
\begin{align*}
e^{\frac{2\zeta^2}{2 + \delta \abs{\zeta}} }\int_{-\zeta}^{\infty} \frac{2}{\mu ( 2 + \abs{\zeta}\delta)}e^{-\frac{2\abs{\zeta}}{ 2 + \abs{\zeta}\delta}y} dy =&\ \frac{1}{\mu \abs{\zeta}} e^{\frac{2\zeta^2}{2 + \delta \abs{\zeta}} }\Big[ -e^{-\frac{2\abs{\zeta}}{ 2 + \abs{\zeta}\delta}y} \Big|_{-\zeta}^{\infty} \Big] = \frac{1}{\mu \abs{\zeta}}.
\end{align*}
 We now bound the second term in \eqref{eq:sa2}. We first show that 
\begin{align*}
g_2(x) \equiv e^{\frac{2}{\delta^2}[\delta x - 2\log(2 + \delta x)]}\int_{x}^{-\zeta} \frac{2}{\mu(2 + \delta y) }e^{\frac{2}{\delta^2}[2\log(2 + \delta y) - \delta y]} dy \leq g_2(0),
\end{align*}
by showing that $g_2'(x) \leq 0$ for all $x \in [0,-\zeta]$. Observe that 
\begin{align*}
g_2'(x) = \frac{2x}{2 + \delta x}g_2(x) - \frac{1}{\mu }\frac{2}{2 + \delta x} = \frac{2}{2 + \delta x}(x g_2(x) - \frac{1}{\mu }), \quad x \in [0,-\zeta].
\end{align*}
Now 
\begin{align*}
0 < x g_2(x) = &\ x e^{\frac{2}{\delta^2}[\delta x-2\log(2 + \delta x)]}\int_{x}^{-\zeta} \frac{2}{\mu(2 + \delta y) }e^{\frac{2}{\delta^2}[2\log(2 + \delta y) - \delta y]} dy\\
\leq&\ x e^{\frac{2}{\delta^2}[\delta x-2\log(2 + \delta x)]}\int_{x}^{-\zeta} \frac{y}{x} \frac{2}{\mu(2 + \delta y) }e^{\frac{2}{\delta^2}[2\log(2 + \delta y) - \delta y]} dy\\
=&\ e^{\frac{2}{\delta^2}[\delta x-2\log(2 + \delta x)]}\int_{x}^{-\zeta}  \frac{2y}{\mu(2 + \delta y) }e^{\frac{2}{\delta^2}[2\log(2 + \delta y) - \delta y]} dy\\
=&\ \frac{1}{\mu } e^{\frac{2}{\delta^2}[\delta x-2\log(2 + \delta x)]}\Big[-e^{\frac{2}{\delta^2}[2\log(2 + \delta y) - \delta y]} \Big|_{x}^{-\zeta} \Big] \leq \frac{1}{\mu },
\end{align*}
where in the last equality we solved the integral by substitution just like in \eqref{eq:integral}. Therefore, $g_2(x) \leq g_2(0)$ for all $x \in [0,-\zeta]$. Now 
\begin{align*}
g_2(0) =&\  e^{\frac{2}{\delta^2}[-2\log(2)]}\int_{0}^{-\zeta} \frac{2}{\mu(2 + \delta y) }e^{\frac{2}{\delta^2}[2\log(2 + \delta y) - \delta y]} dy \\
=&\  e^{\frac{2}{\delta^2}[-2\log(2)]}\int_{0}^{1} \frac{2}{\mu(2 + \delta y) }e^{\frac{2}{\delta^2}[2\log(2 + \delta y) - \delta y]} dy \\
&+ e^{\frac{2}{\delta^2}[-2\log(2)]}\int_{1}^{-\zeta} \frac{2}{\mu(2 + \delta y) }e^{\frac{2}{\delta^2}[2\log(2 + \delta y) - \delta y]} dy.
\end{align*}
For the first integral above, we use the Taylor expansion in \eqref{eq:logtaylor} to see that 
\begin{align*}
e^{\frac{2}{\delta^2}[-2\log(2)]}\int_{0}^{1} \frac{2}{\mu(2 + \delta y) }e^{\frac{2}{\delta^2}[2\log(2 + \delta y) - \delta y]} dy =&\ \int_{0}^{1} \frac{2}{\mu(2 + \delta y) }e^{\frac{2}{\delta^2}[- \frac{2}{2^2} \frac{(\delta y)^2}{2} + \frac{2}{(2 + \xi(\delta y))^3} \frac{(\delta y)^3}{6}]} dy\\
 \leq&\  \frac{1}{\mu } \int_{0}^{1} e^{-\frac{y^2}{2}  + \frac{\delta}{12} } dy \leq \frac{1}{\mu } e^{\delta/12}. 
\end{align*}
For the second integral, observe that 
\begin{align*}
&\ e^{\frac{2}{\delta^2}[-2\log(2)]}\int_{1}^{-\zeta} \frac{2}{\mu(2 + \delta y) }e^{\frac{2}{\delta^2}[2\log(2 + \delta y) - \delta y]} dy \\
\leq&\ \frac{1}{\mu } e^{\frac{2}{\delta^2}[-2\log(2)]}\int_{1}^{-\zeta} \frac{2y}{2 + \delta y }e^{\frac{2}{\delta^2}[2\log(2 + \delta y) - \delta y]} dy\\
=&\ \frac{1}{\mu }e^{\frac{2}{\delta^2}[-2\log(2)]} \Big[-e^{\frac{2}{\delta^2}[2\log(2 + \delta y) - \delta y]} \Big|_{1}^{-\zeta} \Big] \leq \frac{1}{\mu },
\end{align*}
where the equality is obtained by solving the integral just like in \eqref{eq:integral}, and in the last inequality above we used the fact that
\begin{align*}
e^{\frac{2}{\delta^2}[2\log(2 + \delta y) - \delta y]} \leq e^{\frac{2}{\delta^2}[2\log(2)]}, \quad y \geq 0.
\end{align*}
Therefore, 
\begin{align*}
g_2(0) \leq \frac{1}{\mu } \Big(e^{\delta/12} + 1 \Big).
\end{align*}
To recap, we now have the following bound on \eqref{eq:sa2}:
\begin{align*}
\frac{1}{p(x)} \int_{x}^{\infty} \frac{2}{a(y)}p(y) dy \leq&\ \frac{1}{\mu \abs{\zeta}} + g_2(0) \leq \frac{1}{\mu } \Big( \frac{1}{\abs{\zeta}} + e^{\delta/12} + 1 \Big), \quad x \in [0,-\zeta],
\end{align*}
which proves the first part of \eqref{eq:fbound2}. For the second part, we fix $x \geq -\zeta$ and see that
\begin{align*}
&\ \frac{1}{p(x)} \int_{x}^{\infty} \frac{2}{a(y)}p(y) dy \notag = e^{\frac{2\abs{\zeta}}{ 2 + \abs{\zeta}\delta}x}\int_{x}^{\infty}  \frac{2}{\mu ( 2 + \abs{\zeta}\delta)}e^{-\frac{2\abs{\zeta}}{ 2 + \abs{\zeta}\delta}y} dy= \frac{1}{\mu \abs{\zeta}},
\end{align*}
which concludes the proof of \eqref{eq:fbound2}. We now prove \eqref{eq:fbound3}. Fix $x \leq -1/\delta$, then 
\begin{align*} 
\frac{1}{p(x)} \int_{-\infty}^{x} \frac{2\abs{y}}{a(y)}p(y) dy = e^{x^2} \int_{-\infty}^{x} \frac{-2y}{\mu }e^{-y^2} dy = \frac{1}{\mu }.
\end{align*}
Now fix $x \in [-1/\delta, 0]$, then
\begin{align*}
 \frac{1}{p(x)} \int_{-\infty}^{x} \frac{2\abs{y}}{a(y)}p(y) dy =&\ e^{\frac{3}{\delta^2}}e^{\frac{2}{\delta^2}[\delta x - 2\log(2 + \delta x)]}\int_{-\infty}^{-1/\delta} \frac{-2y}{\mu }e^{-y^2} dy \notag \\
&+ e^{\frac{2}{\delta^2}[\delta x - 2\log(2 + \delta x)]}\int_{-1/\delta}^{x} \frac{-2y}{\mu(2 + \delta y) }e^{\frac{2}{\delta^2}[2\log(2 + \delta y) - \delta y]} dy\\
=&\ \frac{1}{\mu } e^{\frac{2}{\delta^2}}e^{\frac{2}{\delta^2}[\delta x - 2\log(2 + \delta x)]}\notag \\
&+ \frac{1}{\mu } e^{\frac{2}{\delta^2}[\delta x - 2\log(2 + \delta x)]}\Big[ e^{\frac{2}{\delta^2}[2\log(2 + \delta y) - \delta y]}\Big|_{-1/\delta}^{x} \Big]\\
=&\ \frac{1}{\mu },
\end{align*}
where in the second equality we computed the integral like in \eqref{eq:integral}. We now fix $x \in [0,-\zeta]$ and observe that
\begin{align*}
\frac{1}{p(x)} \int_{-\infty}^{x} \frac{2\abs{y}}{a(y)}p(y) dy =&\ e^{\frac{3}{\delta^2}}e^{\frac{2}{\delta^2}[\delta x - 2\log(2 + \delta x)]}\int_{-\infty}^{-1/\delta} \frac{-2y}{\mu }e^{-y^2} dy \notag \\
&+ e^{\frac{2}{\delta^2}[\delta x - 2\log(2 + \delta x)]}\int_{-1/\delta}^{0} \frac{-2y}{\mu(2 + \delta y) }e^{\frac{2}{\delta^2}[2\log(2 + \delta y) - \delta y]} dy\notag \\
&+ e^{\frac{2}{\delta^2}[\delta x - 2\log(2 + \delta x)]}\int_{0}^{x} \frac{2y}{\mu(2 + \delta y) }e^{\frac{2}{\delta^2}[2\log(2 + \delta y) - \delta y]} dy\notag \\
=&\ \frac{1}{\mu } e^{\frac{2}{\delta^2}} e^{\frac{2}{\delta^2}[\delta x - 2\log(2 + \delta x)]}  \notag \\
&+ \frac{1}{\mu } e^{\frac{2}{\delta^2}[\delta x - 2\log(2 + \delta x)]}\Big[ e^{\frac{2}{\delta^2}[2\log(2)]}  - e^{\frac{2}{\delta^2}} \Big]\notag \\
&+ \frac{1}{\mu } e^{\frac{2}{\delta^2}[\delta x - 2\log(2 + \delta x)]}\Big[ e^{\frac{2}{\delta^2}[2\log(2)]} -  e^{\frac{2}{\delta^2}[2\log(2 + \delta x) - \delta x]}  \Big] \notag \\
\leq&\ \frac{2}{\mu } e^{\frac{2}{\delta^2}[\delta x - 2\log(2 + \delta x)]}e^{\frac{2}{\delta^2}[2\log(2)]} \\
\leq&\ \frac{2}{\mu }e^{\frac{\zeta^2}{2}},
\end{align*}
where in the second equality we computed the integral like in \eqref{eq:integral}, and in the last inequality we used \eqref{eq:logtaylorineq}. This proves \eqref{eq:fbound3}, and now we prove \eqref{eq:fbound4}. Fix $x \in [0,-\zeta]$ and observe that 
\begin{align*}
&\ \frac{1}{p(x)} \int_{x}^{\infty} \frac{2\abs{y}}{a(y)}p(y) dy \notag \\
=&\ e^{\frac{2}{\delta^2}[\delta x - 2\log(2 + \delta x)]}\int_{x}^{-\zeta} \frac{2y}{\mu(2 + \delta y) }e^{\frac{2}{\delta^2}[2\log(2 + \delta y) - \delta y]} dy\notag \\
&+ e^{\frac{2}{\delta^2}[ 2\log(2 + \delta \abs{\zeta}) - \delta \abs{\zeta}]}e^{\frac{2\zeta^2}{2+\delta\abs{\zeta}}}e^{\frac{2}{\delta^2}[\delta x - 2\log(2 + \delta x)]}\int_{-\zeta}^{\infty}\frac{2y}{\mu ( 2 + \abs{\zeta}\delta)}e^{-\frac{2\abs{\zeta}}{ 2 + \abs{\zeta}\delta}y} dy\notag \\
=&\ \frac{1}{\mu } e^{\frac{2}{\delta^2}[\delta x - 2\log(2 + \delta x)]}\Big[e^{\frac{2}{\delta^2}[2\log(2 + \delta x) - \delta x]}  - e^{\frac{2}{\delta^2}[2\log(2 + \delta \abs{\zeta}) - \delta \abs{\zeta}]}  \Big] \\
&+\frac{1}{\mu } e^{\frac{2}{\delta^2}[ 2\log(2 + \delta \abs{\zeta}) - \delta \abs{\zeta}]}e^{\frac{2\zeta^2}{2+\delta\abs{\zeta}}}e^{\frac{2}{\delta^2}[\delta x - 2\log(2 + \delta x)]}\Big[e^{-\frac{2\zeta^2}{ 2 + \abs{\zeta}\delta}} + \frac{2 + \delta \abs{\zeta}}{2\zeta^2} e^{-\frac{2\zeta^2}{ 2 + \abs{\zeta}\delta}}  \Big],
\end{align*}
where the first term was integrated like in \eqref{eq:integral}, and for the second integral we used integration by parts. The quantity above equals
\begin{align*}
& \frac{1}{\mu } \Big[1  - e^{\frac{2}{\delta^2}[\delta x - 2\log(2 + \delta x)]}e^{\frac{2}{\delta^2}[2\log(2 + \delta \abs{\zeta}) - \delta \abs{\zeta}]}  \Big] \\
&+\frac{1}{\mu  } e^{\frac{2}{\delta^2}[\delta x - 2\log(2 + \delta x)]}e^{\frac{2}{\delta^2}[2\log(2 + \delta \abs{\zeta}) - \delta \abs{\zeta} ]} \Big[1 + \frac{2 + \delta \abs{\zeta}}{2\zeta^2}   \Big] \\
\leq &\ \frac{1}{\mu} + \frac{1}{\mu \zeta^2} + \frac{\delta}{2\mu \abs{\zeta}},
\end{align*}
where the inequality follows from \eqref{eq:monotone}. Now fix $x \geq -\zeta$, then
\begin{align*}
\frac{1}{p(x)} \int_{x}^{\infty} \frac{2\abs{y}}{a(y)}p(y) dy =&\ e^{\frac{2\abs{\zeta}}{ 2 + \abs{\zeta}\delta}x}\int_{x}^{\infty} \frac{2y}{\mu ( 2 + \abs{\zeta}\delta)}e^{-\frac{2\abs{\zeta}}{ 2 + \abs{\zeta}\delta}y} dy\notag \\
=&\ \frac{1}{\mu \abs{\zeta} } e^{\frac{2\abs{\zeta}}{ 2 + \abs{\zeta}\delta}x}\Big[xe^{-\frac{2\abs{\zeta}}{ 2 + \abs{\zeta}\delta}x} + \frac{2 + \delta \abs{\zeta}}{2\abs{\zeta}} e^{-\frac{2\abs{\zeta}}{ 2 + \abs{\zeta}\delta} x}  \Big] \\
=&\  \frac{x}{\mu \abs{\zeta} } + \frac{1}{\mu \zeta^2} + \frac{\delta}{2\mu \abs{\zeta}}.
\end{align*}
This proves \eqref{eq:fbound4}, and we move on to prove \eqref{eq:fbound5}. Recall the form of $r(x)$ from \eqref{eq:rform}. Fix $x \leq -1/\delta$, then \eqref{eq:normcdfbound} implies that 
\begin{align*}
\frac{\abs{r(x)}}{p(x)}\int_{-\infty}^{x}\frac{2}{a(y)} p(y)dy = 2xe^{x^2} \int_{-\infty}^{x}\frac{2}{\mu } e^{-y^2}dy \leq \frac{2}{\mu}.
\end{align*}
For $x \in [-1/\delta,0]$, 
\begin{align*}
\frac{\abs{r(x)}}{p(x)}\int_{-\infty}^{x}\frac{2}{a(y)} p(y)dy =&\ \frac{2\abs{x}}{2 + \delta x} e^{\frac{3}{\delta^2}} e^{\frac{2}{\delta^2}[\delta x - 2 \log(2 + \delta x)] } \int_{-\infty}^{-1/\delta}\frac{2}{\mu} e^{-y^2}dy\\
&+ \frac{2\abs{x}}{2 + \delta x} e^{\frac{2}{\delta^2}[\delta x - 2 \log(2 + \delta x)] } \int_{-1/\delta}^{x}\frac{2}{\mu (2 + \delta y)} e^{\frac{2}{\delta^2}[2 \log(2 + \delta y) - \delta y] }dy\\
\leq&\ \frac{2\abs{x}}{2 + \delta x} e^{\frac{3}{\delta^2}}e^{\frac{2}{\delta^2}[\delta x - 2 \log(2 + \delta x)] } \int_{-\infty}^{-1/\delta}\frac{-y}{\abs{x}}\frac{2}{\mu}  e^{-y^2}dy\\
&+ \frac{2\abs{x}}{2 + \delta x} e^{\frac{2}{\delta^2}[\delta x - 2 \log(2 + \delta x)] } \int_{-1/\delta}^{x}\frac{1}{\abs{x}} \frac{-2y}{\mu (2 + \delta y)} e^{\frac{2}{\delta^2}[2 \log(2 + \delta y) - \delta y] }dy \\
=&\ \frac{2}{2 + \delta x} e^{\frac{3}{\delta^2}}e^{\frac{2}{\delta^2}[\delta x - 2 \log(2 + \delta x)] } \Big[ \frac{1}{\mu } e^{-y^2}\Big|_{-\infty}^{-1/\delta} \Big]\\
&+\frac{2}{2 + \delta x} e^{\frac{2}{\delta^2}[\delta x - 2 \log(2 + \delta x)] } \Big[\frac{1}{\mu } e^{\frac{2}{\delta^2}[2 \log(2 + \delta y) - \delta y]} \Big|_{-1/\delta}^{x} \Big],
\end{align*}
where in the last equality we integrated the second term just like in \eqref{eq:integral}. Therefore, for $x \in [-1/\delta, 0]$, 
\begin{align*}
\frac{\abs{r(x)}}{p(x)}\int_{-\infty}^{x}\frac{2}{a(y)} p(y)dy \leq&\ \frac{1}{\mu }\frac{2}{2 + \delta x} e^{\frac{2}{\delta^2}}e^{\frac{2}{\delta^2}[\delta x - 2 \log(2 + \delta x)] }\\
&+\frac{1}{\mu }\frac{2}{2 + \delta x} e^{\frac{2}{\delta^2}[\delta x - 2 \log(2 + \delta x)] } \Big[ e^{\frac{2}{\delta^2}[2 \log(2 + \delta x) - \delta x]} - e^{\frac{2}{\delta^2}}\Big]\\
\leq&\ \frac{2}{\mu }.
\end{align*}
This proves \eqref{eq:fbound5} and we now prove \eqref{eq:fbound6}. Fix $x \in [0,-\zeta]$, then
\begin{align*}
&\ \frac{\abs{r(x)}}{p(x)}\int_{x}^{\infty}\frac{2}{a(y)} p(y)dy \\
=&\ \frac{2x}{2 + \delta x} e^{ \frac{2}{\delta^2}(2\log(2 + \delta \abs{\zeta}) - \delta \abs{\zeta} )}e^{\frac{2\zeta^2}{2 + \delta \abs{\zeta}} }e^{\frac{2}{\delta^2}[\delta x - 2 \log(2 + \delta x)] } \int_{-\zeta}^{\infty}\frac{2}{\mu(2 + \delta \abs{\zeta})} e^{\frac{-2\abs{\zeta}}{2 + \delta \abs{\zeta}}y}dy\\
&+ \frac{2x}{2 + \delta x} e^{\frac{2}{\delta^2}[\delta x - 2 \log(2 + \delta x)] } \int_{x}^{-\zeta}\frac{2}{\mu (2 + \delta y)} e^{\frac{2}{\delta^2}[2 \log(2 + \delta y) - \delta y] }dy.
\end{align*}
Using \eqref{eq:monotone}, we see that

\begin{align*}
\frac{\abs{r(x)}}{p(x)}\int_{x}^{\infty}\frac{2}{a(y)} p(y)dy \leq&\ \frac{2x}{2 + \delta x} e^{\frac{2\zeta^2}{2 + \delta \abs{\zeta}} } \int_{-\zeta}^{\infty}\frac{2}{\mu(2 + \delta \abs{\zeta})} e^{\frac{-2\abs{\zeta}}{2 + \delta \abs{\zeta}}y}dy\\
&+ \frac{2x}{2 + \delta x} e^{\frac{2}{\delta^2}[\delta x - 2 \log(2 + \delta x)] } \int_{x}^{-\zeta}\frac{2}{\mu (2 + \delta y)} e^{\frac{2}{\delta^2}[2 \log(2 + \delta y) - \delta y] }dy\\
\leq&\  \frac{2x}{2 + \delta x}  \frac{1}{\mu\abs{\zeta}} \\
&+ \frac{2x}{2 + \delta x} e^{\frac{2}{\delta^2}[\delta x - 2 \log(2 + \delta x)] } \int_{x}^{-\zeta}\frac{y}{x} \frac{2}{\mu (2 + \delta y)} e^{\frac{2}{\delta^2}[2 \log(2 + \delta y) - \delta y] }dy\\
=&\  \frac{2x}{2 + \delta x}\frac{1}{\mu \abs{\zeta}} + \frac{1}{\mu }\frac{2}{2 + \delta x} e^{\frac{2}{\delta^2}[\delta x - 2 \log(2 + \delta x)] } \Big[e^{\frac{2}{\delta^2}[2 \log(2 + \delta y) - \delta y]} \Big|_{x}^{-\zeta} \Big],
\end{align*}
where in the last equality we solved the integral just like in \eqref{eq:integral}. Using \eqref{eq:monotone}, we conclude that for $x \in [0,-\zeta]$, 
\begin{align*}
\frac{\abs{r(x)}}{p(x)}\int_{x}^{\infty}\frac{2}{a(y)} p(y)dy \leq&\ \frac{2x}{2 + \delta x}\frac{1}{\mu \abs{\zeta}} + \frac{1}{\mu }\frac{2}{2 + \delta x} \\
\leq&\  \frac{2}{2 + \delta x}\frac{1}{\mu } + \frac{1}{\mu }\frac{2}{2 + \delta x} \leq \frac{2}{\mu }.
\end{align*}
This proves \eqref{eq:fbound6} in the case when $x \in [0,-\zeta]$. We now prove the remaining part of \eqref{eq:fbound6}. Fix $x \geq -\zeta$, and observe that
\begin{align*}
\frac{\abs{r(x)}}{p(x)}\int_{x}^{\infty}\frac{2}{a(y)} p(y)dy =&\ \frac{2\abs{\zeta}}{2 + \delta \abs{\zeta}} e^{\frac{2\abs{\zeta}}{2 + \delta \abs{\zeta}}x} \int_{x}^{\infty}\frac{2}{\mu(2 + \delta \abs{\zeta})} e^{\frac{-2\abs{\zeta}}{2 + \delta \abs{\zeta}}y}dy \leq \frac{2}{\mu }.
\end{align*}
This proves \eqref{eq:fbound6}, and we move on to verify \eqref{eq:fbound7}. Consider the Lyapunov function $V(x) = x^2$, and recall the form of $G_Y$ from \eqref{eq:GY} to see that 
\begin{align*}
G_Y V(x) =&\ 2x\mu (\zeta + (x + \zeta)^-) + 2\mu \Big(1 + 1(x > -1/\delta)\big(1 -  \delta(\zeta + (x + \zeta)^-)\big)\Big).
\end{align*}
Now when $x < -\zeta$, 
\begin{align*}
G_Y V(x) =&\ -2\mu x^2 + 2\mu \big(1 + 1(x > -1/\delta)(1 +\delta x) \big)\\
\leq&\ -2\mu x^2  + 2\mu \delta x 1\big( x \in [0, -\zeta)\big)+ 4\mu \\
=&\ -2\mu x^2 1(x < 0) - 2\mu \big(x^2 - \delta x \big) 1\big( x \in [0, -\zeta)\big) + 4\mu \\
\leq&\ -2\mu x^2 1(x < 0) - \mu \big(x^2 - \delta^2 \big) 1\big( x \in [0, -\zeta)\big) + 4\mu \\
\leq&\ -2\mu x^2 1(x < 0) - \mu x^2 1\big( x \in [0, -\zeta)\big) + \mu \delta^2 + 4\mu,
\end{align*}
and when $x \geq -\zeta$, 
\begin{align*}
G_Y V(x) =&\ - 2x\mu \abs{\zeta} + 2 \delta \mu \abs{\zeta} + 4\mu \\
=&\ - 2\mu \abs{\zeta} ( x - \delta) 1( \abs{\zeta} < \delta) - 2\mu \abs{\zeta} ( x - \delta) 1( \abs{\zeta} \geq \delta)  + 4\mu\\
\leq&\ - 2\mu \abs{\zeta}x 1( \abs{\zeta} < \delta)   + 2\mu \delta^21( \abs{\zeta} < \delta) - 2\mu \abs{\zeta} ( x - \delta) 1( \abs{\zeta} \geq \delta) +  4\mu.
\end{align*}
Therefore, 
\begin{align*}
G_Y V(x) \leq&\ -2\mu x^2 1(x < 0) -\mu x^2 1(x \in [0,-\zeta)) \\
  & - 2\mu \abs{\zeta}x 1( \abs{\zeta} < \delta)1(x \geq -\zeta) - 2\mu \abs{\zeta} ( x - \delta) 1( \abs{\zeta} \geq \delta)1(x \geq -\zeta)\\
  & + 2\mu \delta^21( \abs{\zeta} < \delta) 1(x \geq -\zeta)  + \mu \delta^2 1(x < -\zeta) + 4\mu,
\end{align*}
i.e.\ $G_Y V(x)$ satisfies 
\begin{align*}
G_Y V(x) \leq -f(x) + g(x),
\end{align*}
where $f(x)$ and $g(x)$ are functions from $\R \to \R_+$. By the standard Foster-Lyapunov condition (see for example \cite[Theorem 4.3]{MeynTwee1993b}), this implies that
\begin{align*}
\EE f(Y(\infty)) \leq \EE g(Y(\infty)),
\end{align*}
or
\begin{align*}
&\ 2\EE \big[ (Y(\infty))^2 1(Y(\infty) < 0)\big] + \EE \big[ (Y(\infty))^2 1(Y(\infty) \in [0,-\zeta))\big] \\
&+ 2\abs{\zeta} \EE \big[Y(\infty)1(Y(\infty) \geq -\zeta)\big] 1( \abs{\zeta} < \delta)\\
&+ 2\abs{\zeta} \EE \big[(Y(\infty)-\delta )1(Y(\infty) \geq -\zeta)\big] 1(\abs{\zeta} \geq \delta) \\
 \leq&\ 2\delta^2 + 4,
\end{align*}
from which we can see that
\begin{align*}
\EE \big[Y(\infty)1(Y(\infty) \geq -\zeta)\big] \leq&\ \frac{\delta^2}{\abs{\zeta}} +\frac{2}{\abs{\zeta}} +  \delta.
\end{align*}
Furthermore, by invoking Jensen's inequality we see that
\begin{align*}
\EE \Big[ \big|Y(\infty) 1(Y(\infty) < 0)\big|\Big] \leq&\  \sqrt{\EE \big[ (Y(\infty))^2 1(Y(\infty) < 0)\big]} \\
\leq&\ \sqrt{\delta^2 + 2}, \\
\EE \Big[ \big|Y(\infty) 1(Y(\infty) \in [0,-\zeta))\big|\Big] \leq&\  \sqrt{\EE \big[ (Y(\infty))^2 1(Y(\infty) \in [0,-\zeta))\big]} \\
\leq&\ \sqrt{2\delta^2 + 4}.
\end{align*}
Hence 
\begin{align*}
\EE \big[ \big|Y(\infty)\big|\big] =&\ \EE \Big[ \big|Y(\infty) 1(Y(\infty) < 0)\big|\Big] + \EE \Big[ \big|Y(\infty) 1(Y(\infty) \in [0,-\zeta))\big|\Big] \\
&+ \EE \big[Y(\infty)1(Y(\infty) \geq -\zeta)\big]\\
\leq&\ \sqrt{\delta^2 + 2} + \sqrt{2\delta^2 + 4} + \frac{2+\delta^2}{\abs{\zeta}} + \delta.
\end{align*}
This proves \eqref{eq:fbound7} and concludes the proof of this lemma.
\end{proof}

\section{Proof of Lemma~\ref{lem:eterm}}
\label{app:eterm}
This section is devoted to proving Lemma~\ref{lem:eterm}.  In this entire section, we reserve the variable $x$ to be of the form $x = x_k = \delta (k - R)$, where $k \in \Z_+$. The form of $f_h'''(x)$ in \eqref{eq:hf3} implies that for any $y \in \R$, 
\begin{align}
&\ \abs{f_h'''(y)-f_h'''(x-)} \notag \\
\leq&\ \abs{r'(y)-r'(x-)} \abs{f_h'(y)} + \abs{r'(x-)}\abs{f_h'(x)-f_h'(y)}  \notag \\
&+ \abs{r(y)-r(x)} \abs{f_h''(y)} + \abs{r(x)}\abs{f_h''(x)-f_h''(y)}  \notag \\
&+ \abs{2/a(x) - 2/a(y)} \abs{h'(y)} + \abs{2/a(x)}\abs{h'(x-)-h'(y)} \notag \\
&+ \abs{\frac{2a'(x-)}{a^2(x)} - \frac{2a'(y)}{a^2(y)}} \Big(\abs{h(y)} + \abs{\EE h(Y(\infty))}\Big) + \abs{\frac{2a'(x-)}{a^2(x)}}\abs{h(x)-h(y)}. \label{eq:lipf}
\end{align}
We first state a few auxiliary lemmas that will help us prove Lemma~\ref{lem:eterm}. These lemmas are proved at the end of this section. The first lemma deals with the case when $y \in (x-\delta, x)$.

\begin{lemma} \label{lem:llminus}
Fix $h(x) \in W_2$ with $h(0)=0$, and let $f_h(x)$ be a solution to the Poisson equation \eqref{eq:poisson} that satisfies the conditions of Lemma~\ref{lem:gb}. Recall that $a(x)$ and $r(x)$ are given by \eqref{eq:adef} and \eqref{eq:rform}, respectively. Then there exists a constant $C>0$ (independent of $\lambda, n$, and $\mu$), such that for all $x = x_k = \delta (k - R)$ with $k \in \Z_+$, all $y \in (x-\delta, x)$, and all $n \geq 1, \lambda > 0$, and $\mu > 0$ satisfying $1 \leq R < n $,
\begin{align}
&\ \abs{r'(y)-r'(x-)}\abs{f_h'(y)} + \abs{r'(x-)}\abs{f_h'(x)-f_h'(y)} \leq \frac{C\delta}{\mu }  \Big(1 +  \frac{1}{\abs{\zeta}}\Big)1(x \leq -\zeta), \label{eq:ebound1} \\
&\ \abs{r(y)-r(x)} \abs{f_h''(y)} + \abs{r(x)}\abs{f_h''(x)-f_h''(y)} \notag \\
&\hspace{0.5cm} \leq \frac{C\delta }{\mu }\bigg[(1+\abs{x})\Big(1 +  \frac{1}{\abs{\zeta}}\Big)  1( x\leq -\zeta) + \abs{\zeta}1(x \geq -\zeta + \delta) \bigg], \label{eq:ebound2} \\
&\ \abs{2/a(x) - 2/a(y)} \abs{h'(y)} + \abs{2/a(x)}\abs{h'(x-)-h'(y)}\leq  \frac{C\delta}{\mu}, \label{eq:ebound3} \\
&\ \abs{\frac{2a'(x-)}{a^2(x)} - \frac{2a'(y)}{a^2(y)}} \abs{\EE h(Y(\infty))} +  \abs{\frac{2a'(x-)}{a^2(x)}}\abs{h(x)-h(y)} \notag \\
&\hspace{0.5cm} \leq \frac{C\delta}{\mu }\Big(1 + \frac{1}{\abs{\zeta}} \Big) 1(x \in [-1/\delta + \delta, -\zeta]) \label{eq:ebound4}\\
&\ \abs{\frac{2a'(x-)}{a^2(x)} - \frac{2a'(y)}{a^2(y)}} \abs{h(y)} \leq \frac{C\delta}{\mu } 1(x \in [-1/\delta + \delta, -\zeta])  \label{eq:ebound5}
\end{align}
\end{lemma}
The second lemma deals with the case when $y \in (x, x+\delta)$.
\begin{lemma} \label{lem:llplus}
Consider the same setup as in Lemma~\ref{lem:llminus}, but this time let $y \in (x, x + \delta)$. Then
\begin{align}
&\ \abs{r'(y)-r'(x-)}\abs{f_h'(y)} + \abs{r'(x-)}\abs{f_h'(x)-f_h'(y)} \notag \\
& \hspace{0.5cm} \leq \frac{C\delta}{\mu } \Big[ \Big(1 +  \frac{1}{\abs{\zeta}}\Big) 1(x \leq -\zeta-\delta)+ \frac{1}{\delta} \Big(1 +  \frac{1}{\abs{\zeta}}\Big)1(x\in \{-1/\delta, -\zeta\}) \Big], \label{eq:epbound1} \\
&\ \abs{r(y)-r(x)} \abs{f_h''(y)} + \abs{r(x)}\abs{f_h''(x)-f_h''(y)} \notag \\
&\hspace{0.5cm} \leq \frac{C\delta }{\mu }\bigg[(1+\abs{x})\Big(1 +  \frac{1}{\abs{\zeta}}\Big)  1( x\leq -\zeta-\delta) + \abs{\zeta}1(x \geq -\zeta ) \bigg], \label{eq:epbound2} \\
&\ \abs{2/a(x) - 2/a(y)} \abs{h'(y)} + \abs{2/a(x)}\abs{h'(x-)-h'(y)}\leq  \frac{C\delta}{\mu}, \label{eq:epbound3} \\
&\ \abs{\frac{2a'(x-)}{a^2(x)} - \frac{2a'(y)}{a^2(y)}} \abs{\EE h(Y(\infty))} +  \abs{\frac{2a'(x-)}{a^2(x)}}\abs{h(x)-h(y)} \notag \\
&\hspace{0.5cm} \leq \frac{C\delta}{\mu }\Big(1 + \frac{1}{\abs{\zeta}} \Big)1(x \in [-1/\delta, -\zeta]) \label{eq:epbound4}\\
&\ \abs{\frac{2a'(x-)}{a^2(x)} - \frac{2a'(y)}{a^2(y)}} \abs{h(y)} \leq \frac{C\delta}{\mu} \Big[ 1(x \in [-1/\delta + \delta, -\zeta - \delta]) + \frac{1}{\delta} 1(x \in \{-1/\delta, -\zeta\})\Big]  \label{eq:epbound5}
\end{align}
\end{lemma}
With these two lemmas, the proof of Lemma~\ref{lem:eterm} becomes trivial.
\begin{proof}[Proof of Lemma~\ref{lem:eterm}]
When $y \in (x-\delta, x)$, we just apply \eqref{eq:ebound1}--\eqref{eq:ebound5} from Lemma~\ref{lem:llminus} to \eqref{eq:lipf} to get \eqref{eq:eboundleft}. Similarly, for $y \in (x,x+\delta)$ we apply \eqref{eq:epbound1}--\eqref{eq:epbound5} of Lemma~\ref{lem:llplus} to \eqref{eq:lipf} to get \eqref{eq:eboundright}. This concludes the proof of Lemma~\ref{lem:eterm}.
\end{proof}

\subsection{Proof of Lemma~\ref{lem:llminus}}
\begin{proof}[Proof of Lemma~\ref{lem:llminus}]
Fix $k \in \Z_+$, let $x = x_k = \delta(k - R)$, and fix $y \in (x-\delta, x)$. Throughout the proof we use $C> 0$ to denote a generic constant that may change from line to line, but does not depend on $\lambda, n$, and $\mu$. To prove this lemma we verify \eqref{eq:ebound1}--\eqref{eq:ebound5}, starting with \eqref{eq:ebound1}. Using the form of $r'(x)$ in \eqref{eq:rform}, we see that
\begin{align*}
\abs{r'(y)-r'(x-)} = \abs{r'(y)-r'(x-)} 1(x \in [-1/\delta + \delta, -\zeta]).
\end{align*} 
Furthermore, $r''(u)$ exists for all $u \in (-1/\delta, -\zeta)$, and from \eqref{eq:rform} one can see that
\begin{align*}
r''(u) = \frac{8\delta}{(2 + \delta u)^3} \leq 8\delta, \quad u \in (-1/\delta, -\zeta).
\end{align*}
Therefore, 
\begin{align}
\abs{r'(y)-r'(x-)}\abs{f_h'(y)} \leq&\ \abs{f_h'(y)}1(x \in [-1/\delta + \delta, -\zeta]) \int_{x-\delta}^{x} \abs{r''(u)} du \notag\\
 \leq&\ \frac{C\delta^2}{\mu }\Big(1 +  \frac{1}{\abs{\zeta}}\Big) 1(x \in [-1/\delta + \delta, -\zeta ]), \label{inl:e1}
\end{align}
where in the last inequality we used the gradient bound \eqref{eq:WCder1}. Furthermore, we observe that
\begin{align*}
\abs{r'(x-)} \leq&\ 4 \times 1(x \leq -\zeta),\\
\abs{f_h'(x) - f_h'(y)} \leq&\ \int_{x-\delta}^{x} \abs{f_h''(u)} du \leq \frac{C\delta}{\mu }\Big[\Big(1 +  \frac{1}{\abs{\zeta}}\Big)1(x \leq -\zeta) + \frac{1}{\abs{\zeta}} 1(x \geq -\zeta+\delta)\Big],
\end{align*}
where in the first line we used the form of $r'(x)$ from \eqref{eq:rform}, and in the second line we used the gradient bound \eqref{eq:WCder2}. Recalling that $\delta \leq 1$, we conclude that
\begin{align*}
&\ \abs{r'(y)-r'(x-)}\abs{f_h'(y)} + \abs{r'(x-)}\abs{f_h'(x)-f_h'(y)} \leq \frac{C\delta}{\mu }  \Big(1 +  \frac{1}{\abs{\zeta}}\Big)1(x \leq -\zeta).
\end{align*}
This proves \eqref{eq:ebound1}, and we move on to show \eqref{eq:ebound2}. Observe that
\begin{align}
\abs{r(x)} \leq&\ 2\abs{x}1(x \leq -\zeta) + \abs{\zeta} 1(x \geq -\zeta+\delta), \notag \\
\abs{r(x) - r(y)} \leq&\ \int_{x-\delta}^{x} \abs{r'(u)} du \leq 4\delta 1(x \leq -\zeta), \notag \\
\abs{f_h''(y)} \leq&\ \frac{C}{\mu }\Big[\Big(1 +  \frac{1}{\abs{\zeta}}\Big)1(x \leq -\zeta) + \frac{1}{\abs{\zeta}} 1(x \geq -\zeta+\delta)\Big], \notag \\
\abs{f_h''(x) - f_h''(y)} \leq&\ \int_{x-\delta}^{x} \abs{f_h'''(u)} du \leq \frac{C\delta }{\mu }\Big[\Big(1 + \frac{1}{\abs{\zeta}}\Big)1(x \leq -\zeta) +  1(x \geq -\zeta+\delta)\Big], \label{inl:e2}
\end{align}
where the first two lines above are obtained using the form of $r(x)$ in \eqref{eq:rform}, and in the last two lines we used the gradient bounds \eqref{eq:WCder2} and \eqref{eq:WCder3}. Combining the bounds above proves \eqref{eq:ebound2}, and we move on to prove \eqref{eq:ebound3}. Observe that
\begin{align}
\abs{2/a(x)} \leq&\ 2/\mu, \notag \\
\abs{2/a(x) - 2/a(y)} \leq&\ 2\int_{x-\delta}^{x} \abs{\frac{a'(u)}{a^2(u)}} du \leq \frac{2\delta}{\mu} 1(x \in [-1/\delta + \delta, -\zeta]), \notag\\
\abs{h'(x-)} \leq&\ 1, \quad \text{ and } \quad \abs{h'(x-) - h'(y)} \leq \norm{h''} \abs{x-y} \leq \delta, \label{inl:e3}
\end{align}
where in the first two lines we used the forms of $a(x)$ and $a'(x)$ from \eqref{eq:adef} and \eqref{eq:ap}, and in the last line we used the fact that $h(x) \in W_2$. Combining these bounds proves \eqref{eq:ebound3}, and we move on to prove \eqref{eq:ebound4}. Observe that 
\begin{align*}
\abs{\frac{2a'(x-)}{a^2(x)}} =&\  \abs{\frac{2\delta}{\mu (2+\delta x)^2}} 1(x \in [-1/\delta + \delta, -\zeta]) \leq \frac{2\delta}{\mu}1(x \in [-1/\delta + \delta, -\zeta]),\\
\abs{\frac{2a'(y)}{a^2(y)}} \leq&\ \frac{2\delta}{\mu}1(x \in [-1/\delta + \delta, -\zeta]),\\
\abs{\EE h(Y(\infty))} \leq&\ \EE \big| Y(\infty)\big|, \quad \text{ and } \quad \abs{h(x)-h(y)} \leq \norm{h'} \abs{x-y} \leq \delta,
\end{align*} 
where in the first line we used the forms of $a(x)$ and $a'(x)$ from \eqref{eq:adef} and \eqref{eq:ap}, and in the last line we used the fact that $h(x) \in W_2$. We use the bounds above together with \eqref{eq:fbound7} and the fact that $\delta \leq 1$ to see that
\begin{align}
&\ \abs{\frac{2a'(x-)}{a^2(x)} - \frac{2a'(y)}{a^2(y)}} \abs{\EE h(Y(\infty))} + \abs{\frac{2a'(x-)}{a^2(x)}}\abs{h(x)-h(y)} \notag \\
\leq&\ \abs{\frac{2a'(x-)}{a^2(x)}}\EE \big| Y(\infty)\big| + \abs{\frac{2a'(y)}{a^2(y)}} \EE \big| Y(\infty)\big| + \frac{2\delta^2}{\mu}1(x \in [-1/\delta + \delta, -\zeta]) \notag \\
\leq&\ \frac{C\delta}{\mu}\Big(1 + \frac{1}{\abs{\zeta}} \Big)1(x \in [-1/\delta + \delta, -\zeta]) + \frac{2\delta^2}{\mu}1(x \in [-1/\delta + \delta, -\zeta]) \notag \\
\leq&\ \frac{C\delta}{\mu }\Big(1 + \frac{1}{\abs{\zeta}} \Big)1(x \in [-1/\delta + \delta, -\zeta]), \label{inl:e4}
\end{align}
which proves \eqref{eq:ebound4}. Lastly we show \eqref{eq:ebound5}. Observe that
\begin{align*}
\abs{\frac{2a'(x-)}{a^2(x)} - \frac{2a'(y)}{a^2(y)}} = \abs{\frac{2a'(x-)}{a^2(x)} - \frac{2a'(y)}{a^2(y)}} 1(x \in [-1/\delta + \delta, -\zeta]),
\end{align*} 
and that the derivative of $2a'(u-)/a^2(u)$ exists for all $u \in (-1/\delta, -\zeta)$ and satisfies 
\begin{align*}
\abs{\bigg(\frac{2a'(u)}{a^2(u)} \bigg)'} = \frac{4\delta^2}{\mu (2 + \delta u)^3}, \quad u \in (-1/\delta, -\zeta).
\end{align*}
Recalling that $\abs{h(y)} \leq \abs{y}$, we see that
\begin{align}
\abs{\frac{2a'(x-)}{a^2(x)} - \frac{2a'(y)}{a^2(y)}} \abs{h(y)} \leq&\ 1(x \in [-1/\delta + \delta, -\zeta]) \int_{x-\delta}^{x} \abs{y}\abs{\bigg(\frac{2a'(u)}{a^2(u)} \bigg)'} du  \notag \\
=&\ 1(x \in [-1/\delta + \delta, -\zeta]) \int_{x-\delta}^{x} \frac{4\delta}{\mu (2 + \delta u)^2} \abs{\frac{\delta y}{(2 + \delta u)}} du \notag  \\
\leq&\ 1(x \in [-1/\delta + \delta, -\zeta]) \frac{4\delta}{\mu }\int_{x-\delta}^{x}  \abs{\frac{\delta y}{ (2 + \delta u)}} du \notag \\
\leq&\ 1(x \in [-1/\delta + \delta, -\zeta]) \frac{4\delta}{\mu } \delta(\delta^2+1), \label{inl:e5}
\end{align}
where to obtain the last inequality, we used the fact that $\abs{y-u} \leq \delta$ and $\delta u \geq -1$ to see that
\begin{align*}
\abs{\frac{\delta y}{ (2 + \delta u)}} = \abs{\frac{\delta (y-u) + \delta u}{ (2 + \delta u)}} \leq \delta^2 + \abs{\frac{\delta u}{ 2 + \delta u}} \leq \delta^2 + 1.
\end{align*}
Recalling that $\delta \leq 1$ establishes \eqref{eq:ebound5}, and concludes the proof of this lemma.
\end{proof}

\subsection{Proof of Lemma~\ref{lem:llplus}}
\begin{proof}[Proof of Lemma~\ref{lem:llplus}]
Fix $k \in \Z_+$, let $x = x_k = \delta(k - R)$, and fix $y \in (x,x+\delta)$. Throughout the proof we use $C> 0$ to denote a generic constant that may change from line to line, but does not depend on $\lambda, n$, and $\mu$. The proof for this lemma is very similar to the proof of Lemma~\ref{lem:llminus}. In most cases, the only adjustment necessary to the proof is to consider cases when $x \leq -\zeta - \delta$ and $x \geq -\zeta$, instead of $x \leq -\zeta$ and $x \geq -\zeta + \delta$.
We now verify \eqref{eq:epbound1}--\eqref{eq:epbound5} in order, starting with \eqref{eq:epbound1}. Using the form of $r'(x)$ in \eqref{eq:rform}, we see that
\begin{align*}
\abs{r'(y)-r'(x-)} =&\ \abs{r'(y)-r'(x-)} 1(x \in [-1/\delta + \delta, -\zeta  - \delta]) \\
&+ (\abs{r'(y)}+2) 1(x = -1/\delta) + \abs{r'(x-)} 1(x = -\zeta).
\end{align*} 
Therefore, 
\begin{align*}
&\ \abs{r'(y)-r'(x-)}\abs{f_h'(y)}\\
=&\ \abs{r'(y)-r'(x-)} \abs{f_h'(y)} 1(x \in [-1/\delta + \delta, -\zeta  - \delta]) \\
&+ (\abs{r'(y)}+2)\abs{f_h'(y)}  1(x = -1/\delta) + \abs{r'(x-)} \abs{f_h'(y)}1(x = -\zeta) \\
 \leq&\ \frac{C\delta^2}{\mu }\Big(1 +  \frac{1}{\abs{\zeta}}\Big) 1(x \in [-1/\delta + \delta, -\zeta - \delta]) + \frac{C}{\mu }\Big(1 +  \frac{1}{\abs{\zeta}}\Big) 1(x = -1/\delta) \\
 &+ \abs{r'(x-)} \abs{f_h'(y)}1(x = -\zeta),
\end{align*}
where in the last inequality, the first term is obtained just like in \eqref{inl:e1}, and the second term comes from the gradient bound \eqref{eq:WCder1} and the fact that $\abs{r'(y)} \leq 4$, which can be seen from \eqref{eq:rform}. Now using the gradient bounds \eqref{eq:WCder1} and \eqref{eq:WCder2}, together with the facts that $\abs{r'(\abs{\zeta}-)} \leq 4$ and $\delta \leq 1$, we see that
\begin{align*}
&\ \abs{r'(x-)} \abs{f_h'(y)}1(x = -\zeta) \\
\leq&\ \abs{r'(x-)} \abs{f_h'(x)}1(x = -\zeta) + \abs{r'(x-)} \abs{f_h'(x) - f_h'(y)}1(x = -\zeta) \\
\leq&\ \frac{C}{\mu }\Big(1 +  \frac{1}{\abs{\zeta}}\Big) 1(x = -\zeta) + \abs{r'(x-)} 1(x = -\zeta) \int_{-\zeta}^{-\zeta + \delta} \abs{f_h''(u)} du \\
\leq&\ \frac{C}{\mu }\Big(1 +  \frac{1}{\abs{\zeta}}\Big) 1(x = -\zeta),
\end{align*}
and therefore
\begin{align*}
\abs{r'(y)-r'(x-)}\abs{f_h'(y)} \leq&\ \frac{C\delta}{\mu }\Big(1 +  \frac{1}{\abs{\zeta}}\Big) 1(x \in [-1/\delta + \delta, -\zeta - \delta]) \\
&+ \frac{C}{\mu }\Big(1 +  \frac{1}{\abs{\zeta}}\Big) 1(x \in \{-1/\delta, -\zeta\}).
\end{align*}
Furthermore, 
\begin{align*}
\abs{r'(x-)}\abs{f_h'(x)-f_h'(y)} \leq&\ \abs{r'(x-)} \int_{x}^{x+\delta} \abs{f_h''(u)} du\\
\leq&\ \frac{C\delta}{\mu }\Big(1 +  \frac{1}{\abs{\zeta}}\Big)1(x \leq -\zeta-\delta) + \frac{C\delta}{\mu \abs{\zeta}}1(x = -\zeta) ,
\end{align*}
where in the second inequality we used that $\abs{r'(x)} \leq 4$ and the gradient bound in \eqref{eq:WCder2}. Recalling that $\delta \leq 1$, we can combine the bounds above to see that 
\begin{align*}
&\ \abs{r'(y)-r'(x-)}\abs{f_h'(y)} + \abs{r'(x-)}\abs{f_h'(x)-f_h'(y)} \notag \\
\leq&\ \frac{C\delta}{\mu } \Big[ \Big(1 +  \frac{1}{\abs{\zeta}}\Big) 1(x \leq -\zeta-\delta)+ \frac{1}{\delta} \Big(1 +  \frac{1}{\abs{\zeta}}\Big)1(x\in \{-1/\delta, -\zeta\}) \Big],
\end{align*}
which proves \eqref{eq:epbound1}.

 The proofs for \eqref{eq:epbound2}, \eqref{eq:epbound3}, and \eqref{eq:epbound4}, are nearly identical to the proofs of \eqref{eq:ebound2}, \eqref{eq:ebound3}, and \eqref{eq:ebound4} from Lemma~\ref{lem:llminus}, respectively, and we do not repeat them here. The only differences to note is that \eqref{eq:epbound2} is separated into the cases $x \leq -\zeta - \delta$ and $x \geq -\zeta$, as opposed to \eqref{eq:ebound2} which has $x \leq -\zeta$ and $x \geq -\zeta + \delta$. Likewise, \eqref{eq:epbound4} contains $1(x \in [-1/\delta, -\zeta])$, whereas \eqref{eq:ebound4} contains $1(x \in [-1/\delta + \delta, -\zeta])$.

 Lastly we prove \eqref{eq:epbound5}. From the form of $a'(x)$ in \eqref{eq:ap}, we see that
\begin{align*}
\abs{\frac{2a'(x-)}{a^2(x)} - \frac{2a'(y)}{a^2(y)}} =&\ \abs{\frac{2a'(x-)}{a^2(x)} - \frac{2a'(y)}{a^2(y)}} 1(x \in [-1/\delta + \delta, -\zeta-\delta]) \\
&+ \abs{\frac{2a'(y)}{a^2(y)}} 1(x = -1/\delta) + \abs{\frac{2a'(x-)}{a^2(x-)}} 1(x = -\zeta).
\end{align*}  
We can repeat the argument from \eqref{inl:e5} to get
\begin{align*}
&\abs{\frac{2a'(x-)}{a^2(x)} - \frac{2a'(y)}{a^2(y)}} \abs{h(y)} \\
 \leq&\ \frac{4\delta}{\mu } \delta(\delta^2+1)1(x \in [-1/\delta + \delta, -\zeta - \delta])  +  \abs{\frac{2a'(y)}{a^2(y)}} \abs{y}1(x=-1/\delta)\\
 & +  \abs{\frac{2a'(x-)}{a^2(x-)}}\abs{y}1(x = -\zeta).
\end{align*}
Then using \eqref{eq:ap1}, the form of $a'(x)$ in \eqref{eq:ap}, and the fact that $a(x) \leq 1/\mu$, we can bound the term above by
\begin{align*}
&\ \frac{C\delta}{\mu }1(x \in [-1/\delta + \delta, -\zeta - \delta]) +  \frac{2}{\abs{a(y)}} \abs{\frac{ya'(y)}{a(y)}}1(x=-1/\delta)  \\
&+  \frac{2}{\abs{a(x-)}} \abs{\frac{(\abs{x}+\delta)a'(x-)}{a(x-)}}1(x = -\zeta) \\
\leq&\ \frac{C\delta}{\mu }1(x \in [-1/\delta + \delta, -\zeta - \delta]) + \frac{C}{\mu }1(x=-1/\delta) + \frac{C}{\mu}1(x = -\zeta).
\end{align*}
Hence,
\begin{align*}
\abs{\frac{2a'(x-)}{a^2(x)} - \frac{2a'(y)}{a^2(y)}} \abs{h(y)} \leq&\ \frac{C\delta}{\mu} \Big[ 1(x \in [-1/\delta + \delta, -\zeta - \delta]) + \frac{1}{\delta} 1(x \in \{-1/\delta, -\zeta\})\Big],
\end{align*}
which proves \eqref{eq:epbound5} and concludes the proof of this lemma.
\end{proof}

\section{Probability Metrics}
\label{app:probmet}
Let $W_2$ be the class of functions defined in \eqref{eq:spacew2}, i.e.\ the class of differentiable functions $h(x): \R \to \R$ such that both $h(x)$ and $h'(x)$ belong to $\lipone$. 
For two random variables $U$ and $V$, define their $W_2$ distance to be 
\begin{equation}
  \label{eq:dW}
  d_{W_2}(U, V) = \sup_{h\in {W_2}} \abs{\EE[h(U)] -\EE[h(V)]}.
\end{equation}
When ${W_2}$ in (\ref{eq:dW}) is replaced by 
\begin{displaymath}
  {\cal H}_{K}=\{1_{(-\infty, a]}(x): a\in \R \},
\end{displaymath}
the corresponding distance is the Kolmogorov distance, denoted by $d_K(U, V)$. In this section we prove the following relationship between the $W_2$ and Kolmogorov distances. This lemma is a modified version of \cite[Proposition 1.2]{Ross2011}. 
\begin{lemma}
\label{lem:pmetrics}
Let $U, V$ be two random variables, and assume that $V$ has a density bounded by some constant $C > 0$. If $d_{W_2}(U, V) < 4C$, then 
\begin{equation}
  \label{eq:dwdKbound}
  d_K(U, V) \le 5\Big(\frac{C}{2}\Big)^{2/3} d_{W_2}(U, V)^{1/3}.
\end{equation}
\end{lemma}
We wish to combine Lemma~\ref{lem:pmetrics} with Theorem~\ref{thm:w2}, but to do so we need a bound on the density of $Y(\infty)$.
\begin{lemma} \label{lem:densbound}
Let $\nu(x) : \R \to \R$ be the density of $Y(\infty)$, whose form is given in \eqref{eq:stdden}. Then for all $n \geq 1, \lambda > 0$, and $\mu > 0$ satisfying $1  \geq R < n $, 
\begin{align*}
\nu(x) \leq 4, \quad x \in \R.
\end{align*}
\end{lemma}
Now combining Theorem~\ref{thm:w2} with Lemmas~\ref{lem:pmetrics} and \ref{lem:densbound} implies that $d_K\big(\tilde X(\infty), Y(\infty) \big)$ converges to zero at a rate of $1/R^{1/3}$. However, we believe this rate to be sub-optimal, and that $d_K\big(\tilde X(\infty), Y(\infty) \big)$ actually vanishes at a rate of $1/\sqrt{R}$. This is supported by numerical results in Appendix~\ref{app:numeric}.

\begin{proof}[Proof of Lemma~\ref{lem:pmetrics} ]
Fix $a \in \R$ and let $h(x) = 1_{(-\infty, a]}(x)$. Now fix $\epsilon \in (0, 2)$ and define the smoothed version
\begin{align*}
h_{\epsilon}(x) = 
\begin{cases}
1, \quad &x \leq a, \\
-\frac{2}{\epsilon^2}(x-a)^2 + 1, \quad &x \in [a, a + \epsilon/2], \\
\frac{2}{\epsilon^2}\big[x - (a + \epsilon/2) \big]^2 - \frac{2}{\epsilon}(x-a) + \frac{3}{2}, \quad &x \in [a + \epsilon/2, a + \epsilon], \\
0, \quad &x \geq a + \epsilon.
\end{cases}
\end{align*}
Since we chose $\epsilon < 2$, it is not hard to see that 
\begin{align*}
\abs{h_{\epsilon}'(x)} \leq \frac{4}{\epsilon^2}, \quad \abs{h_{\epsilon}''(x)} \leq \frac{4}{\epsilon^2}, \quad x \in \R,
\end{align*}
where $h_{\epsilon}''(x)$ is interpreted as the left derivative of $h_{\epsilon}'(x)$ when $x \in \{a, a+\epsilon/2, a+\epsilon\}$. Therefore, $\frac{\epsilon^2}{4} h_{\epsilon}(x) \in W_2$. Then 
\begin{align*}
\EE h(U) - \EE h(V) =&\ \EE h(U) - \EE h_{\epsilon}(V) + \EE \big[ h_{\epsilon}(V) - h(V)\big]\\
\leq&\ \EE h_{\epsilon}(U) - \EE h_{\epsilon}(V) + C \int_{a}^{a+\epsilon} h_{\epsilon}(x) dx \\
=&\ \EE h_{\epsilon}(U) - \EE h_{\epsilon}(V) + C\epsilon/2 \\
\leq&\ \frac{4}{\epsilon^2} d_{W_2}(U,V) + C\epsilon /2,
\end{align*}
Choose $\epsilon = \Big(\frac{2d_{W_2}(U,V)}{C}\Big)^{1/3}$, which lies in $(0,2)$ by our assumption that $d_{W_2}(U,V) < 4C$. Then
\begin{align*}
\EE h(U) - \EE h(V) \leq 5\Big(\frac{C}{2}\Big)^{2/3} d_{W_2}(U, V)^{1/3}.
\end{align*}
Using the function $\tilde h_{\epsilon}(x) \equiv h_{\epsilon}(x+\epsilon)$, a similar argument can be repeated to show that 
\begin{align*}
 \EE h(V) - \EE h(U) =&\ \EE h(V) - \EE \tilde h_{\epsilon}(V) + \EE \tilde h_{\epsilon}(V) - \EE h(U) \\ 
\leq&\ 5\Big(\frac{C}{2}\Big)^{2/3} d_{W_2}(U, V)^{1/3},
\end{align*}
concluding the proof.
\end{proof}

\begin{proof}[Proof of Lemma~\ref{lem:densbound}]
One can check that \eqref{eq:stdden} translates into
\begin{align*}
\nu(x) =
\begin{cases}
\frac{a_1}{\mu }e^{-x^2}, \quad x \leq -1/\delta,\\
\frac{a_2}{\mu(2 + \delta x)}e^{\frac{2}{\delta^2} [2\log(2 + \delta x) - \delta x]}, \quad x \in [-1/\delta, -\zeta], \\
\frac{a_3}{\mu(2 + \delta \abs{\zeta})}e^{\frac{-2\abs{\zeta} x }{2+\delta \abs{\zeta}}}, \quad x \geq -\zeta,
\end{cases}
\end{align*}
where the constants $a_1, a_2, a_3$ make the $\nu(x)$ continuous and integrate to one. To prove that $\nu(x)$ is bounded, we need to bound these three constants. We know that 
\begin{align}
&\ a_1 \int_{-\infty}^{-1/\delta} \frac{1}{\mu }e^{-y^2} dy + a_2 \int_{-1/\delta}^{-\zeta} \frac{1}{\mu(2 + \delta y)}e^{\frac{2}{\delta^2} [2\log(2 + \delta y) - \delta y]} dy  \notag \\
&+ a_3 \int_{-\zeta}^{\infty} \frac{1}{\mu(2 + \delta \abs{\zeta})}e^{\frac{-2\abs{\zeta} y }{2+\delta \abs{\zeta}}}dy = 1. \label{eq:densintegral}
\end{align}
We first bound $\nu(x)$ when $x \leq -1/\delta$. Since $a_1$ and $a_2$ are chosen to make $\nu(x)$ continuous at $x = -1/\delta$, we know that $a_1 e^{-1/\delta^2} = a_2 e^{2/\delta^2}$, or $a_2 = a_1 e^{-3/\delta^2}$. Substituting this into \eqref{eq:densintegral}, we see that

\begin{align*}
a_1 \leq \frac{1}{e^{-3/\delta^2}\int_{-1/\delta}^{-\zeta} \frac{1}{\mu(2 + \delta y)}e^{\frac{2}{\delta^2} [2\log(2 + \delta y) - \delta y]} dy} \leq \frac{2\mu }{e^{-3/\delta^2}\int_{-1/\delta}^{0} e^{\frac{2}{\delta^2} [2\log(2 + \delta y) - \delta y]} dy}.
\end{align*}
The derivative of $e^{\frac{2}{\delta^2} [2\log(2 + \delta y) - \delta y]} $ is positive on the interval $[-1/\delta, 0]$. Therefore, on the interval $[-1/\delta, 0]$, this function achieves its minimum at $y = -1/\delta$, implying that $e^{\frac{2}{\delta^2} [2\log(2 + \delta y) - \delta y]}\geq e^{2/\delta^2}$ for $y \in [-1/\delta, 0]$, and 
\begin{align*}
a_1 \leq \frac{2\mu }{e^{-3/\delta^2}\int_{-1/\delta}^{0} e^{\frac{2}{\delta^2} [2\log(2 + \delta y) - \delta y]} dy} \leq \frac{2\mu }{e^{-3/\delta^2}\int_{-1/\delta}^{0} e^{2/\delta^2} dy} = 2 \mu \delta e^{1/\delta^2}.
\end{align*}
Hence, for $x \leq -1/\delta$, 
\begin{align*}
\nu(x) \leq 2 \mu \delta e^{1/\delta^2} \frac{1}{\mu }e^{-x^2} \leq 2\delta  \leq 2,
\end{align*}
where in the last inequality we used the fact that $R \geq 1$, or $\delta \leq 1$. We now bound $\nu(x)$ when $x \in [-1/\delta, -\zeta]$. By \eqref{eq:densintegral}, 
\begin{align*}
a_2 \leq \frac{1}{\int_{-1/\delta}^{-\zeta} \frac{1}{\mu(2 + \delta y)}e^{\frac{2}{\delta^2} [2\log(2 + \delta y) - \delta y]} dy} \leq \frac{2\mu }{\int_{-1/\delta}^{0} e^{\frac{2}{\delta^2} [2\log(2 + \delta y) - \delta y]} dy}.
\end{align*}
Using the Taylor expansion 
\begin{align*}
2\log(2 + \delta y) = 2\log(2) + \frac{2}{2}\delta y - \frac{2}{(2 + \xi(\delta y))^2} \frac{(\delta y)^2}{2},
\end{align*}
where $\xi(\delta y) \in [\delta y, 0]$, we see that
\begin{align*}
e^{\frac{2}{\delta^2} [2\log(2 + \delta y) - \delta y]} =&\ e^{\frac{4}{\delta^2} \log(2)} e^{\frac{2}{\delta^2} \Big[ \frac{-\delta^2 y^2}{(2 + \xi(\delta y))^2} \Big] } \geq e^{\frac{4}{\delta^2} \log(2)}e^{-2y^2}, \quad y \in [-1/\delta, 0].
\end{align*}
Therefore, 
\begin{align*}
a_2 \leq \frac{2\mu }{\int_{-1/\delta}^{0} e^{\frac{2}{\delta^2} [2\log(2 + \delta y) - \delta y]} dy} \leq \frac{2\mu }{e^{\frac{4}{\delta^2} \log(2)} \int_{-1}^{0} e^{-2y^2}dy} = \frac{2\mu  e^{-\frac{4}{\delta^2} \log(2)}}{\int_{-1}^{0} e^{-2y^2}dy},
\end{align*}
where in the second inequality we used the fact that $\delta \leq 1$. Hence, for $x \in [-1/\delta, -\zeta]$, 
\begin{align*}
\nu(x) = \frac{a_2}{\mu(2 + \delta x)}e^{\frac{2}{\delta^2} [2\log(2 + \delta x) - \delta x]} \leq \frac{2 e^{-\frac{4}{\delta^2} \log(2)}e^{\frac{2}{\delta^2} [2\log(2 + \delta x) - \delta x]}}{\int_{-1}^{0} e^{-2y^2}dy} \leq \frac{2}{\int_{-1}^{0} e^{-2y^2}dy} \leq 4,
\end{align*}
where in the second last inequality we used the fact that on the interval $[-1/\delta, -\zeta]$, the function $e^{\frac{2}{\delta^2} [2\log(2 + \delta x) - \delta x]}$ achieves its maximum at $x = 0$. This fact can be checked by differentiating the function.

Lastly, we bound $\nu(x)$ when $x \geq -\zeta$. By \eqref{eq:densintegral}, 
\begin{align*}
a_3 \leq \frac{1}{\int_{-\zeta}^{\infty} \frac{1}{\mu(2 + \delta \abs{\zeta})}e^{\frac{-2\abs{\zeta} y }{2+\delta \abs{\zeta}}}dy} = 2\mu \abs{\zeta}e^{\frac{2\zeta^2}{2+\delta \abs{\zeta}}},
\end{align*}
which means that for $x \geq -\zeta$,
\begin{align}
\nu(x) = \frac{a_3}{\mu(2 + \delta \abs{\zeta})}e^{\frac{-2\abs{\zeta} x }{2+\delta \abs{\zeta}}} \leq \frac{2\abs{\zeta}}{2 + \delta \abs{\zeta}} \leq \abs{\zeta}, \label{eq:nuboundzeta}
\end{align}
which a useful bound only when $\abs{\zeta}$ is small, say $\abs{\zeta} \leq 1$. Now suppose $\abs{\zeta} \geq 1$. Since $\nu(x)$ is continuous at $x = -\zeta$, we have
\begin{align*}
a_2 = a_3 e^{\frac{-2\zeta^2 }{2+\delta \abs{\zeta}}}e^{-\frac{2}{\delta^2} [2\log(2 + \delta \abs{\zeta}) - \delta \abs{\zeta}]}.
\end{align*}
We insert this into \eqref{eq:densintegral} to see that for $x \geq -\zeta$, 
\begin{align*}
a_3 \leq&\ \frac{e^{\frac{2\zeta^2 }{2+\delta \abs{\zeta}}}}{\int_{-1/\delta}^{-\zeta} \frac{1}{\mu(2 + \delta y)}e^{\frac{2}{\delta^2} [2\log(2 + \delta y) - \delta y]}e^{-\frac{2}{\delta^2} [2\log(2 + \delta \abs{\zeta}) - \delta \abs{\zeta}]} dy} \\
\leq&\ \frac{e^{\frac{2\zeta^2 }{2+\delta \abs{\zeta}}}}{\int_{0}^{-\zeta} \frac{1}{\mu(2 + \delta y)} dy} 
\leq \frac{e^{\frac{2\zeta^2 }{2+\delta \abs{\zeta}}}}{\int_{0}^{-\zeta} \frac{1}{\mu(2 + \delta \abs{\zeta})} dy}
= \mu \frac{2 + \delta \abs{\zeta}}{\abs{\zeta}}e^{\frac{2\zeta^2 }{2+\delta \abs{\zeta}}},
\end{align*}
where in the second inequality we used that 
\begin{align*}
e^{\frac{2}{\delta^2} [2\log(2 + \delta y) - \delta y]}e^{-\frac{2}{\delta^2} [2\log(2 + \delta \abs{\zeta}) - \delta \abs{\zeta}]} \geq 1 , \quad y \in [0,-\zeta],
\end{align*}
which is true because the derivative of the function $e^{\frac{2}{\delta^2} [2\log(2 + \delta y) - \delta y]}$ is negative on the interval $[0, -\zeta]$. Therefore, for $x \geq -\zeta$, 
\begin{align*}
\nu(x) = \frac{a_3}{\mu(2 + \delta \abs{\zeta})}e^{\frac{-2\abs{\zeta} x }{2+\delta \abs{\zeta}}} \leq \frac{1}{\abs{\zeta}} e^{\frac{2\zeta^2 }{2+\delta \abs{\zeta}}} e^{\frac{-2\abs{\zeta} x }{2+\delta \abs{\zeta}}} \leq \frac{1}{\abs{\zeta}}.
\end{align*}
Together with \eqref{eq:nuboundzeta}, this implies that $\nu(x) \leq 1$ for $x \geq -\zeta$. This concludes the proof of this lemma.
\end{proof}

\section{Additional Numerical Results}
\label{app:numeric}
In this section we present some numerical results comparing $Y_0(\infty)$ and $Y(\infty)$. Although Theorem~\ref{thm:w2} is only stated in the context of the $W_2$ metric, we show that $Y(\infty)$ is a superior approximation when it comes to estimating the both the probability mass function, and cumulative distribution function.  Let $\{\pi_k\}_{k=0}^{\infty}$ be the distribution of $X(\infty)$. For $k \in \Z_+$ define 
\begin{align*}
\pi^{Y_0}_k =& \Prob\Big(Y_0(\infty) \in \big[\delta(k - R)-\delta/2, \delta (k - R)+\delta/2\big]\Big),\\
\pi^{Y}_k =& \Prob\Big(Y(\infty) \in \big[\delta(k - R)-\delta/2, \delta (k - R)+\delta/2\big]\Big).
\end{align*}

\begin{table}[h]
  \begin{center}
   \begin{tabular}{rcc|ccc }
\multicolumn{3}{c}{$n=5$}& \multicolumn{3}{c}{$n=100$} \\
$R$ & $\sup_{k \in \Z_+} \big|\pi_k - \pi^{Y_0}_k \big|$ & $\sup_{k \in \Z_+} \big|\pi_k - \pi^{Y}_k \big|$   &R & $\sup_{k \in \Z_+} \big|\pi_k - \pi^{Y_0}_k \big|$ & $\sup_{k \in \Z_+} \big|\pi_k - \pi^{Y}_k \big|$  \\
\hline
 3        &   $2.72\times 10^{-2}$  &$5.84\times 10^{-3}$  & 60 & $1.59\times 10^{-3}$ & $2.95\times 10^{-5}$ \\
 4        &  $1.72\times 10^{-2}$  & $2.67\times 10^{-3}$  & 80  & $1.16\times 10^{-3}$ & $1.92\times 10^{-5}$ \\
 4.9        & $2.51\times 10^{-3}$  &$3.54\times 10^{-4}$  & 98  & $3.59\times 10^{-4}$ & $9.81\times 10^{-6}$ \\
 4.95        &  $1.28\times 10^{-3}$  & $1.78\times 10^{-4}$ & 99  & $2.07\times 10^{-4}$ & $5.80\times 10^{-6}$ \\
 4.99        & $2.61\times 10^{-4}$  & $3.62\times 10^{-5}$& 99.98 & $4.71\times 10^{-5}$ & $1.34\times 10^{-6}$ 
  \end{tabular}
\end{center}
  \caption{Approximating the probability mass function of $\tilde X(\infty)$. \label{tabkolm}}
\end{table}

\begin{table}[h!]
  \begin{center}
   \begin{tabular}{rcc|ccc }
\multicolumn{3}{c}{$n=5$}& \multicolumn{3}{c}{$n=100$} \\
$R$ & $d_K(\tilde X(\infty), Y_0(\infty))$ & $d_K(\tilde X(\infty), Y(\infty))$   &R & $d_K(\tilde X(\infty), Y_0(\infty))$ & $d_K(\tilde X(\infty), Y(\infty))$ \\
\hline
 3        &   $1.32\times 10^{-1}$  &$9.27\times 10^{-2}$  & 60 & $3.43\times 10^{-2}$ & $2.58\times 10^{-2}$ \\
 4        &  $8.76\times 10^{-2}$  & $6.41\times 10^{-2}$  & 80  & $2.93\times 10^{-2}$ & $2.23\times 10^{-2}$ \\
 4.9        & $1.32\times 10^{-2}$  &$9.48\times 10^{-3}$  & 98  & $1.03\times 10^{-2}$ & $8.10\times 10^{-3}$ \\
 4.95        &  $6.84\times 10^{-3}$  & $4.84\times 10^{-3}$ & 99  & $5.86\times 10^{-3}$ & $4.53\times 10^{-3}$ \\
 4.99        & $1.41\times 10^{-3}$  & $9.84\times 10^{-4}$& 99.98 & $1.31\times 10^{-3}$ & $9.93\times 10^{-4}$ 
  \end{tabular}

\begin{tabular}{rc | c | c  }
 $n$ & $R$ &$d_K(\tilde X(\infty), Y_0(\infty))$ & $d_K(\tilde X(\infty), Y(\infty))$\\
\hline
5    & 4       &  $8.76 \times 10^{-2}$ & $6.41 \times 10^{-2}$ \\
50   & 46.59   &  $2.60 \times 10^{-2}$ & $2.11 \times 10^{-2}$ \\
500  & 488.94  &  $7.98 \times 10^{-3}$ & $6.48 \times 10^{-3}$  \\
5000 &  4965   &  $2.50 \times 10^{-3}$ & $2.03 \times 10^{-5}$ 
\end{tabular}
  \end{center}
  \caption{Approximating the cumulative distribution function of $\tilde X(\infty)$. As $R$ increases by a factor of $10$ in the second table, both $d_K(\tilde X(\infty), Y_0(\infty))$ decrease by a factor of $\sqrt{10}$. \label{tabkolm}}
\end{table}

\newpage
\bibliography{dai20160208}
\end{document}